\documentclass[12pt]{amsart}
\usepackage{amssymb,amsmath}
\usepackage{geometry}    
\usepackage{mathrsfs}

\usepackage{vmargin}
\setmargrb{1in}{1in}{1in}{1in}

\newtheorem{theorem}{Theorem}[section]
\newtheorem{lemma}[theorem]{Lemma}

\newtheorem{corollary}[theorem]{Corollary}
\theoremstyle{definition}
\newtheorem{definition}{Definition}
\newtheorem{remark}{Remark}

\newtheorem{problem}{Problem}

\begin{document}

\title[Diophantine approximation in prescribed degree]{Diophantine approximation in prescribed degree}

\author{Johannes Schleischitz}
                            
\thanks{Research supported by the Schr\"odinger scholarship J 3824 of the Austrian Science Fund FWF.\\
Department of Mathematics and Statistics, University of Ottawa, Canada  \\ 
johannes.schleischitz@univie.ac.at}

\begin{abstract}
We investigate approximation to a given real number by algebraic numbers and algebraic integers of prescribed degree. 
We deal with both best and uniform approximation,
and highlight the similarities and differences compared
with the intensely studied problem of approximation by algebraic
numbers (and integers) of bounded degree. We
establish the answer to a question of Bugeaud concerning approximation
to transcendental real numbers by quadratic irrational numbers, and  
thereby we refine a result of Davenport and Schmidt from 1967.
We also obtain several new characterizations of Liouville numbers, 
and certain new insights
on inhomogeneous Diophantine approximation. As an auxiliary side result,
we provide an upper bound for the number of 
certain linear combinations of two given
relatively prime integer polynomials with a linear factor.
We conclude with several open problems.
\end{abstract}

\maketitle

{\footnotesize{
{\em Keywords}: exponents of Diophantine approximation, Wirsing's problem, geometry of numbers, continued fractions \\
{\em Math Subject Classification 2010}: 11J13, 11J82, 11R09}}

\section{Introduction}

\subsection{Outline and notation}  \label{out}

The famous Dirichlet Theorem asserts that for any real number $\zeta$
and any parameter $X>1$, the estimate
\begin{equation} \label{eq:gle}
\vert \zeta-\frac{p}{q}\vert \leq X^{-1}q^{-1}
\end{equation}
has a rational solution $p/q$ with $1\leq q\leq X$.
A well-known immediate implication is the fact that for
any $\zeta\in \mathbb{R}\setminus \mathbb{Q}$
there exist infinitely many rational numbers $p/q$ that satisfy
\begin{equation} \label{eq:gle1}
\vert \zeta-\frac{p}{q}\vert \leq q^{-2}.
\end{equation}
We will also refer to a situation as in \eqref{eq:gle}
as {\em uniform approximation}, whereas \eqref{eq:gle1} 
is a result on {\em best approximation}. 
There has been much research on generalizations 
of \eqref{eq:gle} and \eqref{eq:gle1} concerning
approximation to a real number by 
algebraic numbers degree at most $n$, 
for some given positive integer $n$. 
Define the height $H(P)$ of a polynomial $P$ 
as the maximum absolute value among its coefficients. For 
an algebraic real number $\alpha$, define its height by $H(\alpha)=H(P)$ with $P$ the minimal
polynomial of $\alpha$ with coprime integral coefficients.
Wirsing~\cite{wirsing} proposed the following generalization of
\eqref{eq:gle1}.
For any real number $\zeta$ not algebraic of degree at most $n$,
and any $\epsilon>0$, does the inequality
\begin{equation} \label{eq:dd}
\vert \zeta-\alpha\vert \leq H(\alpha)^{-n-1+\epsilon}
\end{equation}
have infinitely many solutions in algebraic real 
numbers $\alpha$ of degree at most $n$? 
Roughly speaking, is any number not algebraic of degree at most $n$
approximable to degree $n+1$ by algebraic numbers 
of degree at most $n$?
Wirsing could only prove the claim for the 
exponents $-(n+3)/2+\epsilon$,
and despite some effort there has been little improvement since. 
The current 
best exponents due to Tishchenko~\cite{Tsi07} is 
still of the form $-n/2-C_{n}$ for constants $C_{n}<4$.
It is well-known that the corresponding uniform claim, as in \eqref{eq:gle},
is in general false when $n\geq 2$, 
see for example~\cite[Theorem~2.4]{bschlei}.
In case of $\zeta$ an algebraic number of degree at least $n+1$,
the claim \eqref{eq:dd} is true as a consequence of Schmidt's Subspace Theorem. 
Schmidt~\cite{512} raised the question whether \eqref{eq:dd} 
can be sharpened by replacing  the
$\epsilon$ in the exponent by some multiplicative constant. 
Observe that this refined version is no
longer guaranteed even for algebraic numbers.
However, Davenport and Schmidt~\cite{davs} already in 1967 
verified the stronger version
in the special case $n=2$. Indeed, they showed 
that for any $\zeta$ not algebraic of degree at most two
and some $c=c(\zeta)>0$, the estimate
\begin{equation} \label{eq:dsc}
\vert \zeta-\alpha\vert \leq cH(\alpha)^{-3}
\end{equation}
has infinitely many rational or quadratic irrational real 
solutions $\alpha$. Any constant 
$c>\frac{160}{9}\max\{ 1,\zeta^{2}\}$ can be chosen.
In this paper we investigate approximation to a real number
by algebraic numbers of prescribed degree $n\geq 2$. This topic, in contrast to bounded degree,
has been rather poorly investigated.
Some results on this topic are due to 
Bugeaud and Teulie~\cite{bugbuch, teu, teuallein}. 
A variant of Wirsing's problem, or Schmidt's version of it, 
studied by Bugeaud~\cite[Problem~23, Section~10.2]{bugbuch}
is to restrict the degree of $\alpha$ equal to $n$.
Bugeaud himself had recently expressed 
doubts~\cite[Problem~2.9.2]{bdraft} on a positive answer, 
even for $n=2$, rooting in the fact that
the claim is in general false for $\alpha$ cubic algebraic integers
instead of quadratic numbers, as shown by Roy~\cite{roy2, roy}. However,
our first new result shows that for $n=2$ the problem does indeed
have an affirmative answer, and thereby we refine the result of Davenport and Schmidt.

\begin{theorem} \label{bugpro}
Let $\zeta$ be a real number not rational or quadratic irrational. Then, for
some effectively computable constant $c=c(\zeta)$, there exist infinitely many quadratic irrational real numbers $\alpha$
for which the inequality \eqref{eq:dsc} is satisfied.
\end{theorem}

It is worth noticing that the constant $c$ we can provide
for Theorem~\ref{bugpro}
will probably be larger than $\frac{160}{9}\max\{ 1,\zeta^{2}\}$ in \eqref{eq:dsc}, however the explicit computation is cumbersome
and we do not attempt to carry it out.
We continue to discuss Theorem~\ref{bugpro} in
Section~\ref{quadro} below 
and there we also provide more new results concerning
approximation by quadratic irrationals.
In Section~\ref{neures} we study approximation to real numbers 
by algebraic numbers of exact degree $n$, for arbitrary $n\geq 2$. 
We propose a problem related to the Wirsing problem
and solve it for $n=3$.
Moreover,
we illustrate the difference between approximation in bounded 
versus exact degree. Indeed, many classical results turn out
to be false when the degree is fixed.
Section~\ref{unifuni} is devoted to approximation to real numbers
by algebraic integers. 
Davenport and Schmidt~\cite{davsh} wrote a pioneering paper on this topic
in 1969, and more recent results can be found 
in~\cite{bugbuch,teu,teuallein} again. We will discuss these
contributions in Section~\ref{unifuni}. It is tempting to believe that
approximation by algebraic integers of 
degree $n+1$ is closely
related to approximation by algebraic numbers of degree $n$. However,  
Roy~\cite{roy,roy2} constructed counterexamples to an intuitive conjecture for cubic integers.
Related results can also be found in the book of Cassels~\cite{cassels}, and the more recent paper~\cite{bula} by
Bugeaud and Laurent, which investigates inhomogeneous approximation in a wide generality. 
Our new contribution to this topic is related to~\cite{bula} and yields a new characterization of Liouville numbers, that are
numbers for which we may choose an arbitrarily large negative exponent 
in the right hand side of \eqref{eq:gle1}. Proofs, unless very short, 
are carried out in Section~\ref{proofs}. For some 
proofs it is convenient to use the
concept of the parametric geometry of numbers introduced 
by Schmidt and Summerer~\cite{ss}.
Finally we will gather several open problems in Section~\ref{openp}.

We enclose some notation which will simplify the formulation of our results. For a ring $R$, we will denote $R_{\leq n}[T]$ the set of polynomials of degree at most $n$ with coefficients 
in $R$, and similarly define $R_{=n}[T]$ and $R_{\geq n}[T]$. 
The most important instances will be
$\mathbb{Z}_{\leq n}[T]$ and $\mathbb{Z}_{=n}[T]$. 
We denote by $\mathbb{A}_{\leq n}$ and $\mathbb{A}_{=n}$ the set of 
non-zero real
algebraic numbers of degree at most $n$ and equal to $n$, respectively. Similarly,
$\mathbb{A}_{\leq n}^{int}$ and $\mathbb{A}_{=n}^{int}$ denote the sets of non-zero real algebraic integers of degree
at most or exactly $n$, respectively.
We will write $A\ll_{.} B$ when $A\leq c(.)B$ for a constant $c$ that may depend on the subscript arguments, 
and $A\ll B$ when the constant $c$ is absolute. 
Moreover $A\asymp_{.} B$ and $A\asymp B$ will be short notation for 
$A\ll_{.} B\ll_{.} A$ and $A\ll B\ll A$, respectively.

\subsection{Classical and new exponents}  \label{bounded}
We will formulate most of our results in terms of classical exponents of Diophantine approximation and certain variations.
We now define all these exponents and discuss their basic properties.

Let
$w_{n}^{\ast}(\zeta)$ and $\widehat{w}_{n}^{\ast}(\zeta)$ respectively denote the
supremum of real numbers $w^{\ast}$ such that the system
\begin{equation}  \label{eq:wgleichse}
H(\alpha) \leq X, \qquad  0<\vert \zeta-\alpha\vert \leq H(\alpha)^{-1}X^{-w^{\ast}},  
\end{equation}
has a solution $\alpha\in \mathbb{A}_{\leq n}$
for arbitrarily large $X$, and all large $X$, respectively. 
Let $w_{=n}^{\ast}(\zeta)$ and $\widehat{w}_{=n}^{\ast}(\zeta)$ be
defined similarly with $\alpha\in \mathbb{A}_{=n}$ instead of $\alpha\in \mathbb{A}_{\leq n}$. We call the exponents without ''hat''
best approximation constants and the ones with ''hat'' uniform
exponents. Indeed, for $n=1$ this concept relates to the largest
possible exponents in \eqref{eq:gle1} and \eqref{eq:gle} for 
given $\zeta$, respectively. 
In this notation, Wirsing's problem asks 
whether $w_{n}^{\ast}(\zeta)\geq n$ holds for any transcendental real number, and 
Theorem~\ref{bugpro} implies $w_{=2}^{\ast}(\zeta)\geq 2$.
We obviously have the relations
\begin{equation}   \label{eq:obviuos}
0\leq w_{=n}^{\ast}(\zeta)\leq w_{n}^{\ast}(\zeta), \qquad 0\leq \widehat{w}_{=n}^{\ast}(\zeta)\leq \widehat{w}_{n}^{\ast}(\zeta).
\end{equation}
The classical exponents moreover satisfy 
\begin{equation} \label{eq:wepper}
w_{1}^{\ast}(\zeta)\leq w_{2}^{\ast}(\zeta)\leq \cdots, 
\qquad 1=\widehat{w}_{1}^{\ast}(\zeta)\leq \widehat{w}_{2}^{\ast}(\zeta)\leq \cdots,
\end{equation}
where the only non-obvious identity $1=\widehat{w}_{1}^{\ast}(\zeta)$ 
due to Khintchine~\cite{kinn} shows that the exponent of $X$ in
\eqref{eq:gle} cannot be improved for
any irrational real $\zeta$.
In contrast, the analogous claims to \eqref{eq:wepper} for exact degree
are in general false, as we will show in Theorem~\ref{ce} below.

Define $w_{n}(\zeta)$ and $\widehat{w}_{n}(\zeta)$ respectively as the supremum of real numbers $w$ for which the system
\begin{equation} \label{eq:sukmini2}
H(P)\leq X, \qquad 0<\vert P(\zeta)\vert \leq X^{-w}
\end{equation}
has a solution $P\in\mathbb{Z}_{\leq n}[T]$, for
arbitrarily large $X$ and all large $X$, respectively.
Similarly, let $w_{=n}(\zeta)$ and $\widehat{w}_{=n}(\zeta)$ be
the supremum of $w\in{\mathbb{R}}$ such that \eqref{eq:sukmini2}
has an irreducible solution $P\in{\mathbb{Z}_{=n}[T]}$
for arbitrarily large values of $X$, and all large $X$, respectively.
The requirement for the polynomials in the definitions
of $w_{=n}(\zeta)$ and $\widehat{w}_{=n}(\zeta)$ to be irreducible is natural, otherwise
we would have trivial equality with the corresponding classic exponents. Indeed, for any
$Q(T)\in \mathbb{Z}_{=d}[T]$ with
$d<n$, the polynomial $\tilde{Q}(T)=T^{n-d}Q(T)$ has degree precisely $n$, the same
height $H(\tilde{Q})=H(Q)$ and satisfies $\tilde{Q}(\zeta)\asymp_{n,\zeta} Q(\zeta)$.

We again have the obvious relations
\begin{equation}  \label{eq:obvious}  
0\leq w_{=n}(\zeta)\leq w_{n}(\zeta), \qquad 0\leq \widehat{w}_{=n}(\zeta)\leq \widehat{w}_{n}(\zeta),   
\end{equation}
and the classical exponents are non-decreasing
\begin{equation} \label{eq:fritz}
w_{1}(\zeta)\leq w_{2}(\zeta)\leq \cdots, \qquad 1=\widehat{w}_{1}(\zeta)\leq \widehat{w}_{2}(\zeta)\leq \cdots.
\end{equation}
The identity $\widehat{w}_{1}(\zeta)=1$ is due
to Khintchine~\cite{kinn} again. 
On the other hand, the estimates in \eqref{eq:fritz} are again in general false for 
the corresponding exponents of exact degree.
Furthermore any Sturmian continued fraction defined as in~\cite{buglaur}
also provides a counterexample for certain indices. Indeed, for these 
numbers we have both $w_{=2}(\zeta)>w_{=3}(\zeta)$ 
and $\widehat{w}_{=2}(\zeta)>\widehat{w}_{=3}(\zeta)$, as follows from the results in~\cite{ichstw}.
See also~\cite[Theorem~2.1 and Theorem~2.2]{ichlondon}. In contrast to the open Wirsing problem,
for every transcendental real number $\zeta$, a multi-dimensional
variant of Dirichlet's Theorem states
\begin{equation}  \label{eq:dochmachen}
w_{n}(\zeta) \geq \widehat{w}_{n}(\zeta)\geq n.
\end{equation}
In fact, for any $X\geq 1$ 
we find $P\in \mathbb{Z}_{\leq n}[T]$ of height 
at most $X$ such that
$\vert P(\zeta)\vert\leq cX^{-n}$, for an explicit constant $c=c(\zeta)$.
The essential problem in the Wirsing conjecture is that a lower
bound for $w_{=n}^{\ast}(\zeta)$ also requires that the
derivative $P^{\prime}(\zeta)$ of the polynomials $P$ 
inducing \eqref{eq:dochmachen} at $\zeta$ are not too small
in absolute value. This is no longer clear when $n\geq 2$.
Again \eqref{eq:dochmachen} turns out to be false for 
the corresponding exponents of exact degree, 
even $\widehat{w}_{=n}(\zeta)=0$
does occur for certain $\zeta$ when $n\geq 2$. 
We will deal with the problem if $w_{=n}(\zeta)\geq n$ holds 
for any $\zeta$.
Finally we point out that ~\cite[Lemma A.8]{bugbuch} links
the exponents in the form
\begin{align} 
w_{n}^{\ast}(\zeta)&\leq w_{n}(\zeta)\leq w_{n}^{\ast}(\zeta)+n-1, 
\qquad\quad \widehat{w}_{n}^{\ast}(\zeta)\leq \widehat{w}_{n}(\zeta)\leq \widehat{w}_{n}^{\ast}(\zeta)+n-1,   \label{eq:bbb}  \\
w_{=n}^{\ast}(\zeta)&\leq w_{=n}(\zeta)\leq w_{=n}^{\ast}(\zeta)+n-1, 
\qquad \widehat{w}_{=n}^{\ast}(\zeta)\leq \widehat{w}_{=n}(\zeta)\leq \widehat{w}_{=n}^{\ast}(\zeta)+n-1. \label{eq:allesfa}
\end{align}
The left inequalities are easy to infer. They use the elementary 
fact that if $P$ is the minimal polynomial of $\alpha$, then
$\vert P(\zeta)\vert=\vert P(\alpha)-P(\zeta)\vert \leq \vert P^{\prime}(z)\vert\cdot \vert \alpha-\zeta\vert$ for some $z$ between $\alpha$ 
and $\zeta$,
but on the other hand $\vert P^{\prime}(z)\vert\ll_{n,\zeta} H(P)$
when $z$ is close to $\zeta$.
Hence $\vert P(\zeta)\vert
\ll_{n,\zeta} H(P)\vert \alpha-\zeta\vert$ and the claims follow. 
The right inequalities are more difficult to show.

\section{Approximation by quadratic irrational numbers}  \label{quadro}

We recall that a transcendental real number is called $U$-number in Mahler's classification
of real numbers when $w_{n}(\zeta)=\infty$ for
some $n\geq 1$. More precisely, $\zeta$
is called $U_{m}$-number when $m$ is the smallest index for which $w_{m}(\zeta)=w_{=m}(\zeta)=\infty$. The
$U_{1}$-numbers are also called Liouville numbers. 
The set of $U$-numbers is the disjoint union 
of the sets of $U_{m}$-numbers over $m\geq 1$.

\begin{theorem} \label{grad2}
Let $\zeta$ be a real number which satisfies $\widehat{w}_{2}(\zeta)>2$. Then we have
\begin{equation} \label{eq:numerouno}
w_{=2}(\zeta)= w_{2}(\zeta), \qquad \widehat{w}_{=2}(\zeta)=\widehat{w}_{2}(\zeta),
\end{equation}
and 
\begin{equation} \label{eq:numerodos}
w_{=2}^{\ast}(\zeta)= w_{2}^{\ast}(\zeta).
\end{equation}
If additionally $\widehat{w}_{2}^{\ast}(\zeta)\geq 2$
holds, we have $\widehat{w}_{=2}^{\ast}(\zeta)=\widehat{w}_{2}^{\ast}(\zeta)$ as well.
Moreover, $\zeta$ is not a $U$-number.
\end{theorem}

We will see in Theorem~\ref{ce} below
that all identities in Theorem~\ref{grad2} are in general false when 
we drop the assumption $\widehat{w}_{2}(\zeta)>2$.
On the other hand, as indicated in Section~\ref{bounded}, there 
are plenty of numbers, 
including Sturmian continued fractions~\cite{buglaur}, that satisfy the hypothesis $\widehat{w}_{2}(\zeta)>2$.   
The last claim of Theorem~\ref{grad2} 
can be regarded as an extension of~\cite[Th\'{e}or\`{e}me~5.3]{adambug} in the special case $n=2$, where the possibility
that $\zeta$ is a $U_{m}$-number for $m\leq n=2$ was not ruled out.
In fact
the semi-effective exponential upper bounds for the growth of the
sequence $(w_{n}(\zeta))_{n\geq 1}$
from~\cite[Th\'{e}or\`{e}me~4.2]{adambug} 
apply to any $\zeta$ which satisfies $\widehat{w}_{2}(\zeta)>2$.
See also~\cite[Corollary~4.6]{ichspec} for a stronger upper
bound for the exponent $w_{3}(\zeta)$ of the form
$w_{3}(\zeta)<15/(\widehat{w}_{2}(\zeta)-2)^{2}$ 
as soon as $\widehat{w}_{2}(\zeta)>2$.
We will discuss generalizations of these phenomena
in Problem~\ref{p3e} in Section~\ref{openp} below.

With Theorem~\ref{grad2} we can simultaneously 
refine $w_{=2}^{\ast}(\zeta)\geq 2$ from Theorem~\ref{bugpro} and
\begin{equation} \label{eq:moshid}
w_{2}^{\ast}(\zeta)\geq \widehat{w}_{2}(\zeta)(\widehat{w}_{2}(\zeta)-1)
\end{equation}
recently discovered by Moshchevitin~\cite[Theorem~2]{nimo}.

\begin{theorem}
For any real number $\zeta$ not rational or quadratic irrational, we have
\begin{equation} \label{eq:nichtalt}
w_{=2}^{\ast}(\zeta)\geq \widehat{w}_{2}(\zeta)(\widehat{w}_{2}(\zeta)-1)\geq 2.
\end{equation}
\end{theorem}

\begin{proof}
When $\widehat{w}_{2}(\zeta)>2$,
we deduce \eqref{eq:nichtalt} from \eqref{eq:numerodos} and \eqref{eq:moshid}.
On the other hand, in case of $\widehat{w}_{2}(\zeta)=2$, the claim becomes $w_{=2}^{\ast}(\zeta)\geq 2$
and is implied by Theorem~\ref{bugpro}.
\end{proof}

We conclude this section with remarks on Theorem~\ref{bugpro}. 
As indicated in Section~\ref{out}, 
a variation of Theorem~\ref{bugpro} concerning approximation by cubic algebraic integers turns out to be false.
Indeed, a non-empty subclass of Roy's extremal numbers introduced
in~\cite{roy} are not approximable
by cubic algebraic integers to the expected order three. More precisely, it was shown in~\cite{roy2} that for some extremal numbers $\zeta$ 
and any $\alpha\in\mathbb{A}_{\leq 3}^{int}$ we have 
\begin{equation} \label{eq:damienr}
\vert \zeta-\alpha\vert  \gg_{\zeta} H(\alpha)^{-\theta}, \qquad \theta=\frac{\sqrt{5}+3}{2}= 2.6180\ldots <3.
\end{equation}
We also refer to Moshchevitin~\cite{mosh} and Roy~\cite{royneu} for a negative answer to a somehow related two-dimensional
problem introduced by W.M. Schmidt~\cite{wschmidt} concerning 
small evaluations of linear forms with sign restrictions.

% anmerkung: 2 jarnik paper die moshchevitin verwendet hat unten in referenz noch mit "%%" zitiert

\section{Approximation in higher prescribed degree} \label{neures}

\subsection{The problem $w_{=n}(\zeta)\geq n$} \label{natu}

In the case of general $n\geq 2$,
we first want to state a
basic observation relating classical 
and new best approximation exponents.

\begin{lemma} \label{mouz}
Let $n\geq 1$ be an integer and $\zeta$ be any real number. 
For the best approximation exponents we have
the identities
\begin{equation}  \label{eq:drar}
w_{n}(\zeta)= \max \{w_{=1}(\zeta),\ldots,w_{=n}(\zeta)\}, \qquad
w_{n}^{\ast}(\zeta)= \max\{ w_{=1}^{\ast}(\zeta),\ldots,w_{=n}^{\ast}(\zeta)\}.
\end{equation}
\end{lemma}

\begin{proof}
The estimates
\[
w_{n}(\zeta)\geq \max \{w_{=1}(\zeta),\ldots,w_{=n}(\zeta)\}, \qquad
w_{n}^{\ast}(\zeta)\geq \max\{ w_{=1}^{\ast}(\zeta),\ldots,w_{=n}^{\ast}(\zeta)\}
\]
are an obvious consequence of \eqref{eq:obvious}, \eqref{eq:fritz}
and  \eqref{eq:obviuos}, \eqref{eq:wepper}, respectively.
The reverse right inequality follows from pigeon hole principle, since clearly infinitely many $\alpha\in\mathbb{A}_{\leq n}$ 
from the definition of $w_{n}^{\ast}(\zeta)$ must have the same degree. The reverse left estimate
follows similarly, when we also take into account that
the polynomials in the definition of $w_{n}(\zeta)$ 
can be chosen irreducible, 
as noticed by Wirsing~\cite[Hilfssatz~4]{wirsing}. 
\end{proof}

Note that the argument of Lemma~\ref{mouz} no longer applies 
to the uniform 
exponents $\widehat{w}_{n}(\zeta), \widehat{w}_{n}^{\ast}(\zeta)$. 
Indeed, the uniform identities analogous to \eqref{eq:drar} 
turn out to be false in general when $n\geq 2$, as we will see
in Theorem~\ref{ce} below.

We discuss the 
problem, motivated by the Dirichlet Theorem \eqref{eq:dochmachen},
if we can fix the degree of infinitely many 
involved polynomials to be exactly
$n$ and still obtain the lower bound $n$. 
In other words we want to know if $w_{=n}(\zeta)\geq n$ holds for 
any integer $n\geq 1$ and
any transcendental real $\zeta$. 
From Lemma~\ref{mouz} and \eqref{eq:dochmachen}
we only know that $w_{=k}(\zeta)\geq n$ for some $1\leq k\leq n$.
The problem seems to be 
reasonably easier than asking for $w_{=n}^{\ast}(\zeta)\geq n$
related to the Wirsing problem as no information on
derivatives of the polynomials is required. 
For $n=2$, Theorem~\ref{bugpro}
and \eqref{eq:allesfa} settles the claim. 
For $n=3$, we can still show the answer
is affirmative.
For $n\geq 4$, we can no longer provide a definite answer.
However, we establish a sufficient criterion concerning the
irreducibility of certain integer polynomials.

\begin{theorem} \label{basned}
Let $\zeta$ be a transcendental real number. We have 
\begin{equation} \label{eq:toeroe1}
w_{=3}(\zeta)\geq \widehat{w}_{3}(\zeta)\geq 3.
\end{equation}
For $n\geq 4$ an integer,
suppose the following claim holds. 
For any $\epsilon>0$ and a constant $c=c(n,\epsilon)$,
and any pair of coprime
polynomials $P,Q\in\mathbb{Z}_{\leq n}$ 
there exist non-zero integers $a,b$ with 
$\max\{\vert a\vert, \vert b\vert\}\leq c\cdot\max\{ H(P),H(Q)\}^{\epsilon}$,
such that $aP+bQ$ is irreducible. Then we have
\begin{equation} \label{eq:toeroe}
w_{=n}(\zeta)\geq \widehat{w}_{n}(\zeta)\geq n.
\end{equation}
\end{theorem}

The proof of \eqref{eq:toeroe1} relies on a proof of the involved irreducibility criterion
in the particular case $n=3$, which is of some interest on its own. 
For the sequel we write
\[
R_{p}(T)=Q(T)+pP(T), \qquad S_{p}(T)=P(T)+pQ(T),
\]
for given integer polynomials $P,Q$ and an integer $p$. 

\begin{theorem} \label{irrpol}
Let $n\geq 2$ be an integer. Let
$P\in\mathbb{Z}_{\leq n-1}[T]$ non-zero and
$Q\in\mathbb{Z}_{=n}[T]$ such that
\begin{itemize}

\item $P,Q$ have no common
linear factor over $\mathbb{Z}[T]$
\item $Q(0)=0$, and
\item  $\frac{Q(T)}{TP(T)}$ is non-constant.

\end{itemize}
Let $X=\max\{ H(P),H(Q)\}$.
Then for any $\delta>0$, there exists a constant
$c=c(n,\delta)>0$ not depending on $P,Q$
such that the number of prime numbers $p$ for which 
either of the polynomials
$R_{p}$ or $S_{p}$ has a linear factor, is less
than $cX^{\delta}$. Thus,
for any $\varepsilon>0$ there exists $d=d(n,\varepsilon)>0$
and a prime number 
$p\leq dX^{\varepsilon}$ for which both
$R_{p}$ and $S_{p}$ have no linear factor.
Moreover, for any prime $p>nX^{n+1}$ the polynomials $R_{p}$ and $S_{p}$ both have no linear factor.
\end{theorem}

\begin{remark}
The theorem is formally true for $n=1$ as well but 
the assumptions on $P,Q$
cannot be satisfied (which must be the case as the
claim is obviously false). Notice also that when $Q(0)=0$,
the other two conditions are clearly necessary for the conclusion. 
We propose the natural generalization
to drop the condition $Q(0)=0$ upon replacing
the third condition on $P,Q$ by the assumption 
that the rational function $Q(T)/P(T)$ is not 
a linear polynomial. Our proof of Theorem~\ref{irrpol}
cannot be modified in a straightforward
way to settle this conjecture. The situation also
becomes more difficult if we allow $P\in\mathbb{Z}_{=n}[T]$.
\end{remark}

The bound for the number of $R_{p}$ and $S_{p}$ with linear factor
can be significantly reduced
if we assume that the constant coefficient $b_{0}$ 
of $P$ and the 
leading coefficient $a_{n}$ of $Q$ both
do not have an untypically large number of divisors $\tau(b_{0}), \tau(a_{n})$. 
More precisely, the proof shows an upper bound
of the form
$\ll_{n,\delta} \tau(b_{0})\tau(a_{n})\log X$ for the number
of $R_{p}$ or $S_{p}$ with linear factor. 
If we assume $\tau(N)\ll \log N$ which is true
on average (see~\cite[Theorem~3.3]{apostol}) for both
$N=a_{n}$ and $N=b_{0}$,
the bound becomes $\ll_{n,\delta} (\log X)^{3}$. 
We emphasize that the first and third 
condition on $P,Q$ in Theorem~\ref{irrpol} are in particular 
satisfied whenever $P$ and $Q$ have no common factor. 

\begin{corollary}
The conclusion of Theorem~\ref{irrpol} holds if we
assume $P\in\mathbb{Z}_{\leq n-1}[T]$ and 
$Q\in\mathbb{Z}_{=n}[T]$ have no common factor and $Q(0)=0$, 
instead of 
the three conditions on $P,Q$ there. 
\end{corollary}

For $n\geq 4$, the imposed condition in Theorem~\ref{basned}
seems to be rather weak as well, however we do not have
a rigorous proof. Our proof of Theorem~\ref{irrpol}
already requires some results from analytic number theory. 
We quote some related irreducibility results.
It was shown by Cavachi~\cite{cav} 
that for given coprime $Q\in\mathbb{Z}_{=n}[T],P\in\mathbb{Z}_{\leq n-1}[T]$, among the polynomials
$S_{p}$ for $p$ a prime number, only finitely many 
are reducible.
Subsequent papers even provided effective
lower bounds for $p$, in dependence of the degrees and heights 
of $P$ and $Q$, such that for any prime $p$ exceeding this bound the 
irreducibility is settled~\cite{cav2},~\cite{bonciocat}. 
See also~\cite{bugagain} for a recent generalization to prime powers. However, the bounds on $p$ are comparable in size with our
bound $nX^{n+1}$, too weak for the 
purpose to prove~\eqref{eq:toeroe} (in fact it
is not hard to see that $R_{p}$ indeed can be irreducible
for some $p\leq X^{1-\epsilon}$ with $\epsilon>0$, start 
with a reducible
polynomial $U\in\mathbb{Z}_{=n}[T]$ and 
$P\in\mathbb{Z}_{\leq n-1}[T]$ coprime to it 
and let $Q=U-pP$ for large $p$). 
Concerning $R_{p}$, 
no irreducibility results of this kind seem available.
We finally remark that if $w_{n}(\zeta)>w_{n-1}(\zeta)$ we obviously
have $w_{n=}(\zeta)=w_{n}(\zeta)\geq n$ in view of \eqref{eq:dochmachen}. 
In particular, for fixed $\zeta$ the claim \eqref{eq:toeroe}
certainly holds for infinitely 
many indices $n$, unless $\zeta$ is a $U_{m}$-number for $m\geq 2$
(the case $m=1$ will be covered by Corollary~\ref{clintonh} below).

\subsection{Particular classes of numbers}
Our upcoming results 
will frequently use special classes of real numbers with
continued fraction expansions of a special form.
We refer to~\cite[Section~1.2]{bugbuch} for an introduction to continued fractions.
Let $w\geq 1$ be a real parameter.
We define $\mathscr{B}_{w}$ as the class of numbers whose
convergents $p_{i}/q_{i}$ in the continued fraction expansion
satisfy $p_{1}=1, q_{1}=2$, and the recurrence relation 
\[
p_{i+2}=M\lfloor p_{i+1}^{w-1}\rfloor+p_{i}, \qquad
q_{i+2}=M\lfloor q_{i+1}^{w-1}\rfloor+q_{i},
\]
where $M\geq 1$ is any positive integer parameter. This is,
apart from some restriction on $M$, the same
class of numbers considered by
Bugeaud~\cite{bug}. He used the equivalent recursive definition
\begin{equation} \label{eq:zetard}
\zeta=[0;2,M\lfloor q_{1}^{w-1}\rfloor,M\lfloor q_{2}^{w-1}\rfloor,\cdots],
\end{equation}
where $q_{1}=2$ and $q_{j}$ are defined as the 
denominator of the $j$-th 
convergent to the real number $\zeta$.
We naturally extend this concept to $w=\infty$ by 
defining $\mathscr{B}_{\infty}$ the set of numbers with
\begin{equation} \label{eq:unendlich}
\lim_{j\to\infty} \frac{\log a_{j+1}}{\log a_{j}}= \infty,
\end{equation}
for $(a_{j})_{j\geq 1}$ the sequence partial quotients associated 
to $\zeta$. Indeed, similar to \eqref{eq:zetard}, this assumption
implies that every convergent is a very good approximation 
to $\zeta$.
The set $\mathscr{B}_{\infty}$ was called {\em strong Liouville numbers} by LeVeque~\cite{leveque}. In fact, all our results below concerning $\mathscr{B}_{\infty}$ remain valid for the wider class of semi-strong Liouville numbers
introduced by Alnia\c{c}ik~\cite{alni}.
Bugeaud~\cite[Corollary~1]{bug} showed that when $w\geq 2n-1$, 
and $M$ is sufficiently large in terms of $n$, 
then any number as in \eqref{eq:zetard} satisfies
\begin{equation} \label{eq:open}
w_{1}(\zeta)=w_{1}^{\ast}(\zeta)=w_{2}(\zeta)=\cdots=w_{n}(\zeta)=w_{n}^{\ast}(\zeta)=w.
\end{equation}
For $n\geq 1$ and $w\in[n,\infty]$, we denote by $\mathscr{D}_{n,w}$ the class of real numbers that satisfy \eqref{eq:open}.
Any set $\mathscr{D}_{n,\infty}$ coincides with the set of 
Liouville numbers.
Our proof of Theorem~\ref{liouspez} below will show that actually one may take any integer $M\geq 1$ for the conclusion \eqref{eq:open} when $w\geq 2n-1$, in 
other words every $\zeta\in \mathscr{B}_{w}$ satisfies \eqref{eq:open} for $w\geq 2n-1$. 
Moreover the claim obviously remains true when $w=\infty$.
Thus we have
\begin{equation} \label{eq:exist}
\mathscr{B}_{w}\subseteq \mathscr{D}_{n,w}, \qquad w\in[2n-1,\infty].
\end{equation}
For smaller parameters $w\in[n,2n-1)$, it has not been yet shown 
that $\mathscr{D}_{n,w}\neq \emptyset$, so our results below
on sets $\mathscr{D}_{n,w}$ are in fact kind of 
conditional when $w<2n-1$.
However, we strongly believe that $\mathscr{D}_{n,w}\neq \emptyset$
always holds, maybe \eqref{eq:exist} even 
extends to $w\in[n,\infty]$.
In this context we refer to the {\em Main Problem} formulated in~\cite[Section~3.4, page~61]{bugbuch}, which if true would 
directly imply this hypothesis. 
The numbers in $\mathscr{D}_{n,w}$ and $\mathscr{B}_{w}$
have biased approximation properties.
Their special structure permits to determine most
exponents of approximation, and it will turn out that
they behave differently with respect to approximation by 
numbers of bounded degree
than by numbers of prescribed degree. 

\subsection{Properties of $\mathscr{B}_{w}$ and $\mathscr{D}_{n,w}$}
\label{neurey}
It has been settled in~\cite[Theorem~5.1]{j3} and
within the proof of~\cite[Theorem~5.6]{bdraft} respectively
that for $\zeta\in\mathscr{D}_{n,w}$ with $w\in[n,\infty]$, the 
classical uniform exponents can be determined as
\begin{equation} \label{eq:kannt}
\widehat{w}_{n}(\zeta)=n, \qquad \widehat{w}_{n}^{\ast}(\zeta)
= \frac{w}{w-n+1}.
\end{equation}
As a first new contribution we determine the uniform 
exponents of prescribed degree $n$ for numbers in $\mathscr{D}_{n,w}$.
For the sequel we agree on $1/\infty=0$ and $1/0=+\infty$. 

\begin{theorem} \label{bs}
Let $n\geq 2$ be an integer, $w\in[n,\infty]$ and $\zeta\in\mathscr{D}_{n,w}$.
Then 
\begin{equation}  \label{eq:erstard}
\widehat{w}_{=n}(\zeta)=\widehat{w}_{=n}^{\ast}(\zeta)=\frac{n}{w-n+1}.
\end{equation}
\end{theorem}

The left identity 
contrasts $\widehat{w}_{n}(\zeta)>\widehat{w}_{n}^{\ast}(\zeta)$
for $\zeta\in\mathscr{D}_{n,w}$ with $w>n$ from \eqref{eq:kannt}.
For our next corollary,
we define, as for example in~\cite{bug}, 
the {\em spectrum} of an exponent of approximation as the set of real values taken by it as the argument
$\zeta$ runs through the set of transcendental real numbers.
We recall some facts.
Metric results by Baker and Schmidt~\cite{baks}
and Bernik~\cite{bernik} imply that the spectra of $w_{n}(\zeta)$ 
and $w_{=n}(\zeta)$ 
equal $[n,\infty]$, and the spectra 
of $w_{n}^{\ast}(\zeta)$ and $w_{=n}^{\ast}(\zeta)$ contain $[n,\infty]$.
(In fact the inclusion of $\{\infty\}$ requires also the existence of $U_{n}$-numbers, see~\cite{leveque}.) Hence
Wirsing's problem is equivalent to asking whether the spectrum of $w_{n}^{\ast}(\zeta)$ is identical to $[n,\infty]$.
Concerning classical uniform exponents,
it is known that the spectrum of $\widehat{w}_{n}$ is 
contained in $[n,\mu(n)]$ with
\begin{equation}  \label{eq:muehka}
\mu(2)=\frac{3+\sqrt{5}}{2},\quad \mu(3)= 3+\sqrt{2},\quad
\mu(n)= n-\frac{1}{2}+\sqrt{n^{2}-2n+\frac{5}{4}}, \quad n\geq 4.
\end{equation}
The lower bounds arise from \eqref{eq:dochmachen}, the
upper bounds are currently best known~\cite[Theorem~2.1]{bschlei}. 
For previously known results see Davenport and Schmidt~\cite{davsh}. 
The bound $\mu(2)$ is optimal, Roy~\cite{roy} proved equality 
$\widehat{w}_{2}(\zeta)=\mu(2)$ for certain
$\zeta$ he called {\em extremal numbers}. 
We refer to~\cite{bdraft} for further references on the 
spectrum of $\widehat{w}_{2}$. For the exponent $\widehat{w}_{n}^{\ast}$,
it follows from \eqref{eq:wepper} and \eqref{eq:allesfa} 
that its spectrum is contained in $[1,\mu(n)]$, and
furthermore we know~\cite{bdraft} that it contains $[1,2-1/n]$.
Similarly, the spectra of the exponents $\widehat{w}_{=n}(\zeta)$ 
and $\widehat{w}_{=n}^{\ast}(\zeta)$ 
are contained in $[0,\mu(n)]$.

If we let $w$ in \eqref{eq:open}
vary in $[2n-1,\infty]$, from Theorem~\ref{bs} and \eqref{eq:exist} we derive new information on the spectra of 
the exponents 
$\widehat{w}_{=n}(\zeta)$ and $\widehat{w}_{=n}^{\ast}(\zeta)$ 
and certain differences.

\begin{corollary} \label{sb}    % zusatzfrage: was wenn zeta U_m number ueber die größeren w_{=,n} aussagbar???
Let $n\geq 2$ be an integer.
The spectra of $\widehat{w}_{=n}(\zeta)$ and $\widehat{w}_{=n}^{\ast}(\zeta)$
both contain the interval $[0,1]$. 
The spectrum of $\widehat{w}_{n}-\widehat{w}_{=n}$ contains $[n-1,n]$ and the spectrum 
of $\widehat{w}_{n}^{\ast}-\widehat{w}_{=n}^{\ast}$ contains $[1-\frac{1}{n},1]$.
\end{corollary}

\begin{proof}
For $w\in[2n-1,\infty]$
consider any $\zeta\in\mathscr{D}_{n,w}$, which is non-empty by \eqref{eq:exist}. 
Combine the identities \eqref{eq:kannt} and \eqref{eq:erstard}.
\end{proof}

From the above we expect
that the spectra of $\widehat{w}_{=n}(\zeta)$ and $\widehat{w}_{=n}^{\ast}(\zeta)$ actually
contain $[0,n]$, similar as for $\widehat{w}_{n}(\zeta)$ 
and $\widehat{w}_{n}^{\ast}(\zeta)$ where
we expect the interval $[1,n]$ to be included. 
Our next result establishes an estimation of $w_{=n}$ for the class of numbers as above.

\begin{theorem} \label{liouspez}
Let $n\geq 2$, $w\in[2n-1,\infty]$ and $\zeta\in\mathscr{B}_{w}$. Then
\begin{equation} \label{eq:holteufel}
w_{=n}(\zeta)\leq \frac{nw}{w-n+1}.
\end{equation}
\end{theorem}

\begin{remark}
Theorem~\ref{liouspez}  
leads to a new proof of~\cite[Corollary~1]{bug}, that is
\eqref{eq:open} for $\mathscr{B}_{w}$ when $w\geq 2n-1$,
and allows for choosing arbitrary $M$ in \eqref{eq:zetard} (which
we required for \eqref{eq:exist}). We further note that we derive explicit constructions of $\zeta$ with prescribed 
exponent $w_{=n}(\zeta)\in[2n-1,\infty]$.
Indeed, if we put $\xi=\sqrt[n]{\zeta}$ for 
$\zeta\in\mathscr{B}_{w}$ with $w\in(2n-1,\infty]$
the deduction of~\cite[Theorem~1]{bug} together
with \eqref{eq:holteufel} leads to $w_{=n}(\xi)=w$.
On the other hand, no new information on the spectra of
$w_{n}(\zeta)$ and $w_{=n}$ in the interval $[n,2n-1)$ is obtained.
\end{remark}

The results above
illustrate the discrepancy between approximation in 
bounded and in fixed degree. We comprize some remarkable 
facts affirming this different behavior.

\begin{theorem} \label{ce}
Let $n\geq 2$ be an integer and $\zeta\in\mathscr{B}_{w}$ for a parameter $w$. 
If $w>n$, then 
\begin{equation} \label{eq:lucia}
\widehat{w}_{=n}(\zeta)<\widehat{w}_{n}(\zeta), \qquad
\widehat{w}_{=n}^{\ast}(\zeta)<\widehat{w}_{n}^{\ast}(\zeta).
\end{equation}
Moreover, for $w\in[2n-1,\infty)$ we have
\begin{equation} \label{eq:bleiben}
\widehat{w}_{n}(\zeta)\geq \widehat{w}_{n}^{\ast}(\zeta) > 1 = \widehat{w}_{=1}(\zeta)= \max\{ \widehat{w}_{=1}(\zeta),\ldots, \widehat{w}_{=n}(\zeta)\},
\end{equation}
such that the uniform inequalities  analogous
to \eqref{eq:drar} are both false in general.

If $w>2n-1$, we have simultaneously the strict inequalities
\begin{equation} \label{eq:allestickln}
w_{=n}(\zeta)<w_{n}(\zeta), \quad w_{=n}^{\ast}(\zeta)<w_{n}^{\ast}(\zeta), \quad  
\widehat{w}_{=n}(\zeta)<\widehat{w}_{n}(\zeta), \quad \widehat{w}_{=n}^{\ast}(\zeta)<\widehat{w}_{n}^{\ast}(\zeta).
\end{equation}
In particular, when $w=\infty$ we have
\eqref{eq:allestickln} simultaneously for all $n\geq 2$.
\end{theorem}

\subsection{Previous results and consequences}

Davenport and Schmidt~\cite{davsh} established a link 
between approximation to a real number $\zeta$ by
algebraic numbers/integers of bounded degree and simultaneous
approximation to successive powers of $\zeta$. For a convenient
formulation of (variants of) their results 
we introduce the exponents of simultaneous 
approximation $\lambda_{n}(\zeta), \widehat{\lambda}_{n}(\zeta)$
defined by Bugeaud and Laurent~\cite{buglaur}. 
They are given as the supremum of real $\lambda$
such that the system
\[
1\leq x\leq X, \qquad \max_{1\leq j\leq n} \vert \zeta^{j}x-y_{j} \vert \leq X^{-\lambda}
\]
has a solution $(x,y_{1},\ldots,y_{n})\in\mathbb{Z}^{n+1}$ for 
arbitrarily large $X$, and all large $X$, respectively. Dirichlet's~Theorem
implies $\lambda_{n}(\zeta)\geq \widehat{\lambda}_{n}(\zeta)\geq 1/n$.
Khintchine's transference principle~\cite{khint} links the exponents $w_{n}$ and $\lambda_{n}$ in the form
\begin{equation} \label{eq:kinschin}
\frac{w_{n}(\zeta)}{(n-1)w_{n}(\zeta)+n}\leq \lambda_{n}(\zeta)\leq \frac{w_{n}(\zeta)-n+1}{n}.
\end{equation}
See German~\cite{ogerman} for inequalities linking the uniform exponents.
Upper bounds 
for $\widehat{\lambda}_{n}(\zeta)$ and $\lambda_{n}(\zeta)$, respectively,
translate into lower bounds for 
$w_{=n}^{\ast}(\zeta)$ and $\widehat{w}_{=n}^{\ast}(\zeta)$, respectively.

\begin{theorem}[Davenport, Schmidt, Bugeaud, Teulie]  \label{korola}
Let $n\geq 1$ be an integer and
$\zeta$ be a real number not algebraic of degree at most $n/2$.
Assume that there
exist constants $\lambda>0$ and $c>0$, such that for certain arbitrarily large $X$, the estimate
\begin{equation} \label{eq:ko}
1\leq x\leq X, \qquad \max_{1\leq j\leq n} \vert x\zeta^{j}-y_{j}\vert \leq cX^{-\lambda}
\end{equation}
has no solution in an integer vector $(x,y_{1},\ldots,y_{n})$.
Then the inequality
\begin{equation} \label{eq:neugebauer} 
\vert \zeta-\alpha\vert \ll_{n,\zeta} H(\alpha)^{-1/\lambda-1}
\end{equation}
has infinitely many solutions $\alpha\in\mathbb{A}_{=n}$.
Similarly, if \eqref{eq:ko} has no integral solution for all large $X$, then 
\begin{equation} \label{eq:formuni}
H(\alpha)\leq X, \qquad \vert \zeta-\alpha\vert \ll_{n,\zeta} X^{-1/\lambda-1} 
\end{equation}
has a solution $\alpha\in\mathbb{A}_{=n}$ for all large $X$. In particular, we have 
\begin{equation} \label{eq:jodad}
w_{=n}^{\ast}(\zeta)\geq \frac{1}{\widehat{\lambda}_{n}(\zeta)}, \qquad \widehat{w}_{=n}^{\ast}(\zeta)\geq \frac{1}{\lambda_{n}(\zeta)}.
\end{equation}
\end{theorem}

We omit the proof as the results are essentially known
and consequence
of the proofs of Davenport and Schmidt~\cite[Lemma~1]{davsh}
and a slight variant of it by Bugeaud~\cite[Theorem~2.11]{bugbuch}. 
See also the comment subsequent 
to the proof of~\cite[Theorem~2.11]{bugbuch}, and
Bugeaud and Teulie~\cite{teu},~\cite{teuallein}.

The uniform exponents $\widehat{\lambda}_{n}$ involved 
in Theorem~\ref{korola} can be effectively bounded from above for all 
transcendental $\zeta$,
and lead to lower bounds for $w_{=n}^{\ast}(\zeta)$ of size roughly $n/2$, for large $n$. For slight improvements see
Roy~\cite{damiroy} for $n=3$, Laurent~\cite{lau} 
for odd $n\geq 5$ and the more 
recent~\cite{ichrelation},~\cite[Section~4.2]{ichspec}
for even $n\geq 4$. For small
values of $n$ the resulting numerical bounds when
combined with Theorem~\ref{korola} become
\[
w_{=3}^{\ast}(\zeta) \geq 2.3557\ldots, \qquad
w_{=4}^{\ast}(\zeta) \geq 
\frac{4}{\sqrt{73}-7}= 2.5906\ldots,\qquad w_{=5}^{\ast}(\zeta) \geq 
3.
\]

Theorem~\ref{korola} will in fact be a crucial ingredient 
for the proofs of many of our new results. Below
we present some of its immediate consequences when
combined with some recent results from~\cite{j3} and our new results
from Section~\ref{neurey}. 
First we derive a new characterization of
Liouville numbers.

\begin{corollary} \label{najor}
A real number  $\zeta$ is a Liouville number if and only if
\begin{equation} \label{eq:zweitard}
\widehat{w}_{=n}(\zeta)=\widehat{w}_{=n}^{\ast}(\zeta)=0, 
\qquad\quad \text{for any} \;\; n\geq 2.
\end{equation}
In fact, if $\zeta$ is not a Liouville number, then 
$\widehat{w}_{=n}(\zeta)\geq \widehat{w}_{=n}^{\ast}(\zeta)>0$
for any $n\geq 2$.
\end{corollary}

\begin{proof}
If $\zeta$ is a Liouville number then \eqref{eq:zweitard} follows from Theorem~\ref{bs} with $w=\infty$.
If otherwise $w_{1}(\zeta)=\lambda_{1}(\zeta)<\infty$, then \eqref{eq:jodad} yields 
$\widehat{w}_{=n}^{\ast}(\zeta)\geq \lambda_{n}(\zeta)^{-1}\geq \lambda_{1}(\zeta)^{-1}>0$.
\end{proof}

The implication \eqref{eq:zweitard} for Liouville numbers might
appear strong at frist view, but is somehow 
suggestive given the results on inhomogeneous approximation 
by Bugeaud and Laurent~\cite{bula}, see Section~\ref{unifuni} below.
Problem~\ref{prob} in Section~\ref{openp} below asks for a similar characterization involving
the classical exponents $\widehat{w}_{n}^{\ast}(\zeta)$. 
Our second corollary to Theorem~\ref{korola} 
proves a strengthened version of Wirsing's conjecture for numbers with large irrationality exponent.

\begin{corollary} \label{clintonh}
Let $\zeta$ be a real number and $n\geq 1$ an integer. Assume $w_{1}(\zeta)\geq n$ holds. Then
we have $w_{=n}^{\ast}(\zeta)\geq n$.
\end{corollary}

\begin{proof}
It was shown in~\cite[Theorem~1.12]{j3} that $w_{1}(\zeta)\geq n$ implies $\widehat{\lambda}_{n}(\zeta)=1/n$.
Hence the assertion is derived from Theorem~\ref{korola}.
\end{proof}

Observe Corollary~\ref{clintonh} applies in particular to all numbers in any class $\mathscr{D}_{n,w}$ for $w\geq n$.
Our last corollary establishes some more exponents for strong Liouville numbers.

\begin{corollary}  \label{referenzk}
Let $\zeta\in\mathscr{B}_{\infty}$. 
Then $w_{=n}(\zeta)=w_{=n}^{\ast}(\zeta)=n$ holds for all $n\geq 2$.
\end{corollary}

\begin{proof}
From Corollary~\ref{clintonh} and \eqref{eq:bbb} we know that $w_{=n}(\zeta)\geq w_{=n}^{\ast}(\zeta)\geq n$.
On the other hand, \eqref{eq:holteufel} with $w=\infty$ implies $w_{=n}(\zeta)\leq n$ for $n\geq 2$.
\end{proof}

We remark that we cannot expect the result 
to extend for arbitrary Liouville numbers.
For the formulation of our final result in this section
we need to define successive minima exponents.
For $1\leq j\leq n+1$, define $w_{n,j}(\zeta)$ and $\widehat{w}_{n,j}(\zeta)$ respectively 
as the supremum of $w$ such that \eqref{eq:sukmini2}
has $j$ linearly independent solutions for arbitrarily large $X$ and all large $X$, respectively.
We see that $w_{n,1}(\zeta)=w_{n}(\zeta)$ and $\widehat{w}_{n,1}(\zeta)=\widehat{w}_{n}(\zeta)$.
Mahler showed that the identities
\begin{equation} \label{eq:mahler}
\lambda_{n}(\zeta)^{-1}=\widehat{w}_{n,n+1}(\zeta), \qquad
\widehat{\lambda}_{n}(\zeta)^{-1}=w_{n,n+1}(\zeta)
\end{equation}
are valid for any transcendental real $\zeta$. 
These are special cases of Mahler's duality, see 
Schmidt and Summerer~\cite{ss} 
and also~\cite[(1.24)]{j2} for more general versions. 
We show that in general we cannot replace the right hand sides 
$1/\widehat{\lambda}_{n}(\zeta)=w_{n,n+1}(\zeta)$ and 
$1/\lambda_{n}(\zeta)=\widehat{w}_{n,n+1}(\zeta)$ 
of \eqref{eq:jodad} respectively, by the next 
larger successive minimum
value $w_{n,n}(\zeta)$ and $\widehat{w}_{n,n}(\zeta)$, respectively. 

\begin{theorem} \label{prof}
Let $n\geq 1$ be an integer and $w>2n-1$. 
For $\zeta\in\mathscr{D}_{n,w}$ we have
\[
\widehat{w}_{=n}^{\ast}(\zeta)=\widehat{w}_{=n}(\zeta) < \widehat{w}_{n,n}(\zeta).
\]
For $\zeta\in\mathscr{B}_{w}$ moreover
\[
w_{=n}^{\ast}(\zeta)=w_{=n}(\zeta)<w_{n,n}(\zeta).
\]
\end{theorem}

It is not clear whether the analogous inequalities for the 
classic exponents can be satisfied.

\section{Approximation by algebraic integers} \label{unifuni}

We define several new variants of the classical exponents, related to the approximation to a real number by algebraic integers.

\begin{definition}  \label{def2}
Let $\zeta$ be a real number and $n\geq 1$ an integer.
Let $w_{n}^{int}(\zeta)$ (and $w_{=n}^{int}(\zeta)$ resp.) be the supremum of $w$ such that
\eqref{eq:sukmini2} has a monic polynomial solution $P\in \mathbb{Z}_{\leq n}$
(and an irreducible monic solution $P\in \mathbb{Z}_{=n}$ resp.) for arbitrarily large $X$.  
Similarly, define $\widehat{w}_{n}^{int}(\zeta)$ (and $\widehat{w}_{=n}^{int}(\zeta)$ resp.) 
as above, with the respective properties satisfied for all large $X$.
Denote by $w_{n}^{\ast int}(\zeta)$ (and $w_{=n}^{\ast int}(\zeta)$ resp.) the supremum of $w^{\ast}$ such that 
\eqref{eq:wgleichse}
has a solution $\alpha\in \mathbb{A}_{\leq n}^{int}$ (and $\alpha\in \mathbb{A}_{=n}^{int}$ resp.) for arbitrarily large $X$.
Similarly, define $\widehat{w}_{n}^{\ast int}(\zeta)$ (and $\widehat{w}_{=n}^{\ast int}(\zeta)$ resp.)
as above, with the respective properties satisfied for 
all large $X$.
\end{definition}

By a similar argument as in~\cite[Hilfssatz~4]{wirsing} we may consider only irreducible polynomials within 
the definition of $w_{n}^{int}(\zeta)$. On the other hand,
we do not expect this be true for the uniform
exponents $\widehat{w}_{n}^{int}(\zeta)$,
although we do not address the topic of counterexamples here.
The irreducibility assumption on the polynomials
with respect to the exponents of prescribed degree again avoids trivial identities, 
as in Section~\ref{bounded}.
The corresponding versions of the obvious relations
\eqref{eq:obviuos}, \eqref{eq:wepper}, \eqref{eq:obvious}, 
\eqref{eq:drar}, \eqref{eq:fritz}, \eqref{eq:bbb} 
and \eqref{eq:allesfa} hold again, 
apart from $w_{1}^{int}(\zeta)=\widehat{w}_{1}^{int}(\zeta)=
w_{1}^{\ast int}(\zeta)=\widehat{w}_{1}^{\ast int}(\zeta)=0$
unless $\zeta\in\mathbb{Z}$. 
The monotonicity conditions
will most likely again require bounded degree,
however again we do not address counterexamples in exact degree.
We should also notice the obvious facts
\[
w_{n}(\zeta)\geq w_{n}^{int}(\zeta), \quad \widehat{w}_{n}(\zeta)\geq \widehat{w}_{n}^{int}(\zeta),
\quad w_{n}^{\ast}(\zeta)\geq w_{n}^{\ast int}(\zeta), \quad \widehat{w}_{n}^{\ast}(\zeta)\geq \widehat{w}_{n}^{\ast int}(\zeta).
\]
However, approximation by elements in $\mathbb{A}_{\leq n}^{int}$ should rather be compared 
to approximation by elements in $\mathbb{A}_{\leq n-1}$, as there is the same degree of freedom in the choice
of coefficients for the corresponding minimal polynomials.

We quote a variant of Theorem~\ref{korola}, again
essentially due to Davenport and Schmidt.

\begin{theorem}[Davenport, Schmidt] \label{ppp}
Let $m,n$ be positive integers with $m\geq n+1$, 
and $\zeta$ be a real number not algebraic 
of degree at most $n/2$. Assume that there
exist constants $\lambda>0$ and $c>0$ such that 
for arbitrarily large values of $X$, the estimate \eqref{eq:ko} has no solution in an integer vector $(x,y_{1},\ldots,y_{n})$.
Then, the inequality
\begin{equation} \label{eq:neugebauerrr} 
\vert \zeta-\alpha\vert \ll_{m,\zeta} H(\alpha)^{-1/\lambda-1}
\end{equation}
has infinitely many solutions $\alpha\in \mathbb{A}_{=m}^{int}$.
In particular, we have 
\begin{equation} \label{eq:jodadrr}
w_{=m}^{\ast int}(\zeta)\geq \frac{1}{\widehat{\lambda}_{n}(\zeta)}, \qquad w_{n+1}^{\ast int}(\zeta)\geq \frac{1}{\widehat{\lambda}_{n}(\zeta)}.
\end{equation}
Similarly, if \eqref{eq:ko} has no solutions for all large $X$, then 
\begin{equation} \label{eq:formunirr}
H(\alpha)\leq X, \qquad \vert \zeta-\alpha\vert \ll_{m,\zeta} X^{-1/\lambda-1} 
\end{equation}
has a solution $\alpha\in \mathbb{A}_{=m}^{int}$ for all large $X$. In particular we have 
\begin{equation} \label{eq:dregsrr}
\widehat{w}_{=m}^{\ast int}(\zeta)\geq \frac{1}{\lambda_{n}(\zeta)}, \qquad \widehat{w}_{n+1}^{\ast int}(\zeta)\geq \frac{1}{\lambda_{n}(\zeta)}.
\end{equation}
\end{theorem}

The claims \eqref{eq:neugebauerrr} and \eqref{eq:jodadrr} reproduce~\cite[Lemma~1]{davsh},
see also~\cite[Theorem~2.11]{bugbuch} and~\cite{teuallein}. The dual claims are obtained similarly, we will omit the proof.
The claim is closely related to the very general main 
result in~\cite{bula} 
by Bugeaud and Laurent on inhomogeneous approximation, of which we will discuss a special case below.
Similarly to Theorem~\ref{korola}, known estimates 
for $\widehat{\lambda}_{n}$ lead
to lower bounds roughly of size $n/2$ for $w_{=n+1}^{\ast int}$.

Recall \eqref{eq:damienr} holds for special numbers. On the other hand,
it is unknown and was posed as a problem in~\cite{bugbuch} and recently rephrased in \cite{bdraft},
whether $w_{n+1}^{\ast int}(\zeta)\geq n$
holds for any transcendental real $\zeta$ when $n\geq 3$. The analogue problem 
for $w_{n+1}^{int}(\zeta)$ is open as well. Again both answers are positive for a pair $n,\zeta$ with
the property $\widehat{\lambda}_{n}(\zeta)=1/n$, 
in view of Theorem~\ref{ppp}. We notice the answer is also positive
when $\zeta$ allows sufficiently good rational approximations, analogously to Corollary~\ref{clintonh}.

\begin{corollary} \label{clintonhi}
Let $\zeta$ be a real number and $n\geq 1$ an integer. Assume $w_{1}(\zeta)\geq n$ holds. Then
we have $w_{=m}^{\ast int}(\zeta)\geq n$ for any $m\geq n+1$. In particular $w_{n+1}^{\ast int}(\zeta)\geq n$.
\end{corollary}

As Corollary~\ref{clintonh}, the claim follows directly 
from~\cite[Theorem~1.12]{j3} and Theorem~\ref{ppp}.
The main contribution of this paper concerning approximation by algebraic integers
are bounds for the uniform constants $\widehat{w}_{n}^{int}(\zeta)$
and $\widehat{w}_{n}^{\ast int}(\zeta)$, for 
special numbers $\zeta$.
Another characterization of Liouville numbers is obtained as a special case.

\begin{theorem} \label{algint}
Let $n\geq 2$ be an integer, $w\in[n,\infty]$ and 
$\zeta\in\mathscr{D}_{n,w}$. Then
\begin{equation} \label{eq:lebal}
\frac{n-1}{w-n+2} \leq 
\widehat{w}_{n}^{\ast int}(\zeta) \leq
\widehat{w}_{n}^{int}(\zeta) \leq \frac{n}{w-n+1}.
\end{equation}
In particular, a transcendental real number $\zeta$ is a Liouville number if and only if
\begin{equation} \label{eq:propella}
\widehat{w}_{n}^{int}(\zeta)=\widehat{w}_{n}^{\ast int}(\zeta)= 0, \qquad\quad \text{for any} \;\; n\geq 1.
\end{equation}
\end{theorem}

We point out that Theorem~\ref{algint} can be interpreted 
in terms of inhomogeneous approximation,
complementing~\cite{bula}.
Indeed, \eqref{eq:propella} yields that 
for $n\geq 1 ,\epsilon>0$, $\zeta$ a Liouville number and $\alpha\in\{\zeta^{n+1},\zeta^{n+2},\ldots\}$ (and more generally
any $\alpha=Q(\zeta)$ for $Q\in\mathbb{Q}_{\geq n+1}[T]$), 
the system
\begin{equation} \label{eq:zuwachs}
\max_{0\leq j\leq n} \vert x_{j}\vert \leq X, \qquad  \vert \alpha+x_{0}+\zeta x_{1}+\cdots+\zeta^{n}x_{n}\vert \leq X^{-\epsilon}
\end{equation}
has no solution for certain arbitrarily large values of $X$. 
The main result in~\cite{bula} shows the same for Lebesgue almost
all $\alpha$. The latter provided a major
improvement on Cassels~\cite[Theorem~3 of Chapter III]{cassels}. Thus our contribution in \eqref{eq:propella}
can be interpreted as to provide explicit examples of $\alpha$ for which the metric claim is satisfied. 
The metric result in~\cite{bula} and the proof of \eqref{eq:lebal} below furthermore suggest equality in the two left inequalities in \eqref{eq:lebal} for any real $\zeta$.
Note that if otherwise $\alpha=Q(\zeta)$ 
for $Q\in\mathbb{Q}_{\leq n}[T]$, then for 
all large $X$ and certain $x_{i}$ the right expression in \eqref{eq:zuwachs} is $\ll_{Q}X^{-n}$, as it is 
roughly speaking just a shift of the homogeneous problem.

\section{Proofs} \label{proofs}

The proof of Theorem~\ref{bugpro} is based on the result of Davenport and Schmidt \eqref{eq:dsc}. To rule out that all good approximations
to $\zeta$ are rational we use the method
from~\cite[Theorem~1.12]{j3}, which we explicitly carry out again
for the reason to be self-contained and the convenience of the reader. Recall $\Vert \alpha\Vert$ denotes the distance of 
$\alpha\in\mathbb{R}$ to the nearest integer. 

\begin{proof}[Proof of Theorem~\ref{bugpro}]
Assume for given $\zeta$ the claim would be false. Then, 
by the result of Davenport and Schmidt, for some 
constant $c=c(\zeta)$
there exist infinitely many rational numbers $\alpha=y_{0}/x_{0}$ for which \eqref{eq:dsc} holds (in particular 
$\lambda_{1}(\zeta)\geq 2$). Now we essentially follow 
the proof of~\cite[Theorem~1.12]{j3} for $n=2$.
For a fraction $y_{0}/x_{0}$ as above, we clearly may assume
$\vert\zeta-y_{0}/x_{0}\vert\leq 1$, and thus
the formula $\vert \zeta^{2}-y_{0}^{2}/x_{0}^{2}\vert=\vert\zeta-y_{0}/x_{0}\vert \cdot \vert\zeta+y_{0}/x_{0}\vert\leq
(2\vert \zeta\vert+1)\cdot \vert\zeta-y_{0}/x_{0}\vert$ implies
\begin{equation} \label{eq:wiedavor}
\left\vert \zeta^{j}-\frac{y_{0}^{j}}{x_{0}^{j}}\right\vert 
\leq c_{1} x_{0}^{-3}, \qquad j\in{\{1,2\}},
\end{equation}
for $c_{1}=\max\{ (2\vert \zeta\vert+1)c,1\}$.
Define $X=x_{0}^{2}/(2c_{1})$ and let
$1\leq x\leq X$ be an arbitrary integer. 
Since $x\leq x_{0}^{2}/(2c_{1})\leq x_{0}^{2}/2<x_{0}^{2}$,
the integer $x$ has a representation in base $x_{0}$ as
\[
x= b_{0}+b_{1}x_{0}, \qquad b_{i}\in\{0,1,2,\ldots,x_{0}-1\}.
\]
Denote by $i\in\{0,1\}$ the smallest index with $b_{i}\neq 0$, and
further let $u=i+1\in\{1,2\}$.
Since $x_{0},y_{0}$ are coprime and $b_{i}\neq 0$, we have
\begin{equation} \label{eq:zzz}
\left\Vert x\frac{y_{0}^{u}}{x_{0}^{u}}\right\Vert = 
\left\Vert b_{i}x_{0}^{u-1}\frac{y_{0}^{u}}{x_{0}^{u}}\right\Vert= \left\Vert \frac{b_{i}y_{0}^{u}}{x_{0}}\right\Vert\geq x_{0}^{-1}.
\end{equation}
On the other hand \eqref{eq:wiedavor} yields
\begin{equation} \label{eq:zz}
\left\vert x\left(\zeta^{u}-\frac{y_{0}^{u}}{x_{0}^{u}} \right)\right\vert\leq X\left\vert\zeta^{u}-\frac{y_{0}^{u}}{x_{0}^{u}} \right\vert
\leq \frac{x_{0}^{2}}{2c_{1}}\cdot c_{1} x_{0}^{-3}= \frac{1}{2}x_{0}^{-1}.
\end{equation}
Combination of \eqref{eq:zzz} and \eqref{eq:zz} and the triangular inequality give
\[
\max\{ \Vert \zeta x\Vert, \Vert\zeta^{2} x\Vert\} \geq \Vert \zeta^{u}x\Vert \geq \frac{1}{2}x_{0}^{-1}= c^{\prime} X^{-1/2},
\]
for the constant $c^{\prime}=1/\sqrt{8c_{1}}$ that again 
depends on $\zeta$ only. 
Thus, since $x\leq X$ was arbitrary, the assumption \eqref{eq:ko} of Theorem~\ref{korola} is satisfied for 
$\lambda=1/2$ and the constant $c^{\prime}$. Hence \eqref{eq:neugebauer} applies, which yields precisely the claim.
Since $c$ in \eqref{eq:dsc} and thus $c_{1}$
is effective
and the implied constant in \eqref{eq:neugebauer} can be
made effective as well, so is our constant.
\end{proof}

For the proof of Theorem~\ref{grad2} we recall the notion of best approximation polynomials of 
a given degree $n$ associated to a real number $\zeta$. 
It can be defined as the sequence of integer 
polynomials $(P_{i})_{i\geq 1}$ with the properties
$1\leq H(P_{1})\leq H(P_{2})\leq \cdots$ and $\vert P_{i}(\zeta)\vert$ 
minimizes the value $\vert P(\zeta)\vert$ among $P\in \mathbb{Z}_{\leq n}[T]$ 
of height $0<H(P)\leq H(P_{i})$.
The polynomials involved in the definition of $w_{n}$ can obviously be chosen as best approximation polynomials.
Furthermore every best approximation polynomial satisfies $\vert P_{i}(\zeta)\vert\ll_{n,\zeta} H(P_{i})^{-n}$ by Dirichlet's~Theorem,
see also the proof of~\cite[Lemma~8.1]{bugbuch}.
Moreover $\vert P_{i}(\zeta)\vert\leq H(P_{i})^{-\widehat{w}_{n}(\zeta)+\epsilon}$ for any $\epsilon>0$ and 
sufficiently large $i\geq i_{0}(\epsilon)$.
We will utilize the estimates
\begin{equation} \label{eq:hoehen}
H(P_{1}P_{2})\asymp_{n} H(P_{1})H(P_{2})
\end{equation}
for any polynomials $P_{1},P_{2}\in \mathbb{Z}_{\leq n}[T]$, 
sometimes referred to as Gelfond's Lemma. 
See also~\cite{wirsing} or \cite[Lemma~A.3]{bugbuch}.
We will apply~\cite[Theorem~5.1]{j3} for $n=2$ several times, which asserts
that $w_{1}(\zeta)\geq n$ implies $\widehat{w}_{n}(\zeta)=n$,
or equivalently $\widehat{w}_{n}(\zeta)>n$ implies $w_{1}(\zeta)<n$.

\begin{proof}[Proof of Theorem~\ref{grad2}]
First we show \eqref{eq:numerouno}.
In view of the obvious inequalities \eqref{eq:obvious}, it suffices to show $w_{=2}(\zeta)\geq w_{2}(\zeta)$
and $\widehat{w}_{=2}(\zeta)\geq \widehat{w}_{2}(\zeta)$.
Note that from our assumption $\widehat{w}_{2}(\zeta)>2$
and~\cite[Theorem~5.1]{j3} we infer $w_{1}(\zeta)<2$. 
Hence, since any quadratic best approximation polynomial satisfies $\vert P(\zeta)\vert\ll_{n,\zeta} H(P)^{-2}$,
no linear polynomial of large height can induce a quadratic best approximation polynomial. 
Moreover, essentially by \eqref{eq:hoehen}, also no product $P=P_{1}P_{2}$ of linear polynomials $P_{i}$
of large enough height $H(P)$ can be a best approximation. 
Indeed, if $\epsilon>0$ and we write
$H(P_{1})H(P_{2})=:H$, we have $H(P)\gg H$ by \eqref{eq:hoehen} but also
\[
 \vert P(\zeta)\vert= \vert P_{1}(\zeta)\vert\cdot \vert P_{2}(\zeta)\vert
\geq H(P_{1})^{-w_{1}(\zeta)-\epsilon}H(P_{2})^{-w_{1}(\zeta)-\epsilon}\gg H^{-w_{1}(\zeta)-\epsilon}.
\]
If we choose $\epsilon=(2-w_{1}(\zeta))/2>0$ we again obtain a contradiction to $P$ being a best approximation polynomial.
Thus any quadratic best approximation polynomial of sufficiently large height 
is irreducible of degree two. The deduction of \eqref{eq:numerouno} is now obvious.
Next we show \eqref{eq:numerodos}. Let $(\alpha_{i})_{i\geq 1}$ be a sequence of rational or
quadratic irrational numbers as in the definition of
$w_{2}^{\ast}(\zeta)$,  with minimal polynomials $P_{i}$ 
respectively. By Theorem~\ref{bugpro}, we can assume
$\vert\zeta-\alpha_{i}\vert \ll H(\alpha_{i})^{-3}$.
With the standard estimate $\vert P_{i}(\zeta)\vert \ll_{n,\zeta} H(P_{i})\vert \zeta-\alpha_{i}\vert$ 
mentioned already in Section~\ref{bounded},
we infer $\vert P_{i}(\zeta)\vert\ll_{\zeta} H(P_{i})^{-2}$.
If infinitely many among the polynomials $P_{i}$ were linear, we would have $w_{1}(\zeta)\geq 2$
and hence again by~\cite[Theorem~5.1]{j3} we infer $\widehat{w}_{2}(\zeta)=2$, contradicting the assumption.
Hence all but finitely many $\alpha_{i}$ are quadratic irrational and \eqref{eq:numerodos} follows. The 
claim on $\widehat{w}_{2}^{\ast}$ follows similarly.

For the last claim, observe that it was shown 
in~\cite[Th\'{e}or\`{e}me~5.3]{adambug} that 
when $\widehat{w}_{n}(\zeta)>n$
and $\zeta$ is a $U$-number, then it must be a $U_{m}$-number 
for $m\leq n$. Applied for $n=2$,
we have to exclude that $\zeta$ is a $U_{1}$-number or a $U_{2}$-number.
Now \cite[Theorem~5.1]{j3} implies directly that $\zeta$ cannot be a $U_{1}$-number. Similarly,
for $\zeta$ any $U_{2}$-number, we obtain $\widehat{w}_{2}(\zeta)=2$ from~\cite[Corollary~2.5]{bschlei}, contradicting our hypothesis.
\end{proof}

We turn to the proofs of Section~\ref{neures}. We start with the
proof of the polynomial criterion. We recall some notation and classical facts from analytic number theory. Let $N\geq 1$ be an integer. 
The Prime Number Theorem 
implies that there are $\pi(N)\gg N/\log N$ primes up to $N$, and
the number of divisors $\tau(N)$ of $N$ is bounded by 
$\tau(N)\ll_{\epsilon} N^{\epsilon}$
for arbitrarily small $\epsilon>0$, see
the book of Apostol~\cite[page~296]{apostol}.
Finally, the number of prime divisors $\omega(N)$ of
$N$ is bouned by $\omega(N)\ll \log N$
(more precisely $\omega(N)\ll \log N/\log\log N$ with asymptotic equality
when $N$ is primorial, see Hardy and Wright~\cite{hw}).

\begin{proof}[Proof of Theorem~\ref{irrpol}]
Let $P,Q$ as in the theorem with $X= \max\{ H(P),H(Q)\}$, 
and $\delta\in(0,1)$ be given.
We may write
\begin{equation} \label{eq:heute}
Q(T)= a_{1}T+a_{2}T^{2}+\cdots+a_{n}T^{n}, \qquad 
P(T)=b_{0}+b_{1}T+\cdots +b_{n-1}T^{n-1},
\end{equation}
with $b_{0}, a_{n}$ non-zero.
Indeed, $b_{0}\neq 0$ since $P$ has no common linear factor 
with $Q$ and $T\vert Q(T)$, and $a_{n}\neq 0$ since $Q$ has 
exact degree $n$ by assumption.

We first show the bounds for $R_{h}$ with $h$ a prime.
Assume that $h$ is prime
and the polynomial $R_{h}(T)=Q(T)+hP(T)$  
has a linear factor $q_{h}T-p_{h}$. Equivalently each
$R_{h}$ has a rational root $p_{h}/q_{h}$, written 
in lowest terms.
Inserting $R_{h}(p_{h}/q_{h})=0$ in \eqref{eq:heute}
and multiplication with $q_{h}^{n}\neq 0$ yields
\begin{equation} \label{eq:det}
a_{1}p_{h}q_{h}^{n-1}+a_{2}p_{h}^{2}q_{h}^{n-2}+\cdots+a_{n}p_{h}^{n}
+h(b_{0}q_{h}^{n}+b_{1}p_{h}q_{h}^{n-1}+\cdots+b_{n-1}p_{h}^{n-1}q_{h})
=0.
\end{equation}
Since any expression apart from $a_{n}p_{h}^{n}$ contains the 
factor $q_{h}$
and $(p_{h},q_{h})=1$, we conclude $q_{h}\vert a_{n}$. 
Similarly $p_{h}\vert (hb_{0})$. 
Since $a_{n}\neq 0$, the quoted result from~\cite{apostol} above 
with $\epsilon=\delta/3$ yields that it has at most
$\tau(a_{n})\ll_{\delta} \vert a_{n}\vert^{\delta/3}$ divisors, so there appear at most 
$\ll_{\delta} \vert a_{n}\vert^{\delta/3}\leq H(Q)^{\delta/3}\leq X^{\delta/3}$ different denominators $q_{h}$. For any
such fixed $q=q_{h}$, we estimate the number of primes 
$h$ for which the polynomial $R_{h}$ can have a root $p_{h}/q$. 

Recall $p_{h}\vert (hb_{0})$, such that 
either $p_{h}\vert b_{0}$ or $p_{h}=hs_{h}$ 
for some $s_{h}\vert b_{0}$.
We treat both possible cases separately. Assume first
$p_{h}\vert b_{0}$. Since $b_{0}\neq 0$,
there are at most 
$\tau(b_{0})\ll_{\delta} \vert b_{0}\vert^{\delta/3}\leq 
H(Q)^{\delta/3}\leq X^{\delta/3}$
divisors of $b_{0}$, so there arise at most
$\ll_{\delta} X^{\delta/3}$ different $p_{h}$. However,
for given $p_{h},q=q_{h}$ the number $h$ is uniquely 
determined by \eqref{eq:det}. Indeed, otherwise both expressions
$P(p_{h}/q_{h})$ and $Q(p_{h}/q_{h})$
must vanish, contradicting 
the hypothesis that $P,Q$ have no common linear factor.
Hence also
the number of $h$ belonging to this class is $\ll_{\delta} X^{\delta/3}$.

Now assume $p_{h}=hs_{h}$ with $s_{h}\vert b_{0}$. 
Since $b_{0}\neq 0$ also $s_{h}\neq 0$. Then 
for our fixed $q_{h}=q$, the relation \eqref{eq:det} becomes
after division by $h\neq 0$ and rearrangements
\begin{equation} \label{eq:hungrig}
h^{n}(a_{n}s^{n}+b_{n-1}s^{n-1}q)+\cdots+
h(a_{2}s^{2}q^{n-2}+sb_{1}q^{n-1})+a_{1}sq^{n-1}+b_{0}q^{n}=0.
\end{equation}
Reducing modulo $h$ we see that $h\vert (q^{n-1}(sa_{1}+b_{0}q))$.
Since $h\vert p_{h}$ and $(p_{h},q)=1$ we cannot 
have $h\vert q$, such that
\begin{equation} \label{eq:teile}
h\vert (s_{h}a_{1}+b_{0}q).
\end{equation}
First assume $s_{h}a_{1}+b_{0}q\neq 0$.
Since $q$ is fixed and $s_{h}\vert b_{0}\neq 0$, there are at most 
$\tau(b_{0})\ll_{\delta} \vert b_{0}\vert^{\delta/3}\leq X^{\delta/3}$
choices for $s_{h}$, consequently 
there arise at most $\ll_{\delta} X^{\delta/3}$ different
numbers $N_{s}=s_{h}a_{1}+b_{0}q$. 
Since $q\vert a_{n}$ and $s_{h}\vert b_{0}$, all
quantities $s_{h},a_{1},b_{0},q$ are bounded by $X$
and thus $\vert N_{s}\vert\leq 2X^{2}$ for all $N_{s}\neq 0$.
Hence each $N_{s}\neq 0$ has at most 
$\omega(N_{s})\ll \log \vert N_{s}\vert\leq \log(2X^{2})\ll \log X$ 
prime divisors. So by \eqref{eq:teile}
at most $\ll_{\delta} X^{\delta/3} \log X$ choices for $h$ arise 
in this way.

Now assume $s_{h}a_{1}+b_{0}q=0$. Then from \eqref{eq:hungrig} 
after division by $hs_{h}\neq 0$ we obtain 
that $h\vert (q^{n-2}(a_{2}s_{h}+b_{1}q))$
and again since $h\nmid q$ we conclude $h\vert (a_{2}s_{h}+b_{1}q)$.
In case of $s_{h}a_{2}+b_{1}q\neq 0$, with the same argument 
as above we again obtain at most 
$\ll_{\delta} X^{\delta/3} \log X$ choices for $h$. 
If otherwise $s_{h}a_{2}+b_{1}q=0$,
then we proceed as above to obtain $h\vert (a_{3}s_{h}+b_{2}q)$. 
Repeating this procedure up to $h\vert (a_{n}s_{h}+b_{n-1}q)$
leads to in total at most 
$\ll_{n,\delta} X^{\delta/3} \log X$ choices of $h$, unless we have
$s_{h}a_{k}+b_{k-1}q=0$ for all $1\leq k\leq n$. 
Since $a_{n}s_{h}\neq 0$, the latter implies
$(a_{1},a_{2},\ldots,a_{n})=-\frac{q}{s_{h}}(b_{0},b_{1},\ldots,b_{n-1})$, in other words
$Q(T)/(TP(T))=-q/s_{h}$ is constant, which 
we excluded by assumption.

Thus we indeed only have $\ll_{n,\delta} X^{\delta/3} \log X$ choices of 
primes $h$ for which $R_{h}$ has a root
with given denominator $q_{h}=q$. Since we have noticed
that at most $\ll_{\delta} X^{\delta/3}$ different $q$ occur
and $\log X\ll_{\delta} X^{\delta/3}$,
the total number of $h$ for which $R_{h}$ has a rational 
root is indeed $\ll_{n,\delta} X^{\delta}$.

Similar arguments show that the number of primes $l$ 
for which $S_{l}=P+lQ$ has a linear factor are of order 
$\ll_{n,\delta} X^{\delta}$ as well. We only sketch the proof.
Inserting $S_{l}(p_{l}/q_{l})=0$ in \eqref{eq:heute} we obtain
$p_{l}\vert b_{0}$ and hence only $\ll_{n,\delta} X^{\delta/3}$ 
many numerators $p_{l}$ of roots $p_{l}/q_{l}$
of $S_{l}$ can appear. We again estimate the number 
of primes $l$ with fixed $p_{l}=p$. 
Since $q_{l}\vert (lha_{n})$, by a very 
similar recursive procedure as for $R_{h}$, 
the number of such $l$ is again of
order $\ll_{n,\delta} X^{\delta/3}\log X$
unless $a_{n}/b_{n-1}=a_{n-1}/b_{n-2}=\cdots=a_{1}/b_{0}$,
which is again excluded by assumption. The claim follows as above.

By Prime Number Theorem there 
are $\gg Y/\log Y\gg_{\delta} Y^{\delta}$
primes up to $Y$, if we choose $Y=dX^{\varepsilon}$ 
for $\varepsilon=2\delta$ and suitable $d=d(n,\delta)$
we avoid the primes in both
sets of cardinality $\ll_{n,\delta} X^{\delta}$ above.

Finally we show that $R_{h}$ or $S_{l}$
having a linear factor implies $h\leq nX^{n+1}$ or $l\leq nX^{n+1}$,
respectively. We only show the claim for $R_{h}$, the other case 
is very similar again. Recall that
$\vert q_{h}\vert \leq \vert b_{0}\vert \leq X$ 
since $q_{h}\vert b_{0}$.
We showed that either $p_{h}\vert b_{0}$ or $p_{h}=hs_{h}$
with $s_{h}\vert b_{0}$. In the first case 
$\vert p_{h}\vert \leq X$,
such that if we express $h$ from \eqref{eq:det} 
as a rational function then
each expression in the numerator $q_{h}^{n}Q(p_{h}/q_{h})$
and denominator $q_{h}^{n}P(p_{h}/q_{h})$ is bounded 
in absolute value by $X^{n+1}$. Hence the numerator 
has absolute value at most
$nX^{n+1}$, and the denominator at least $1$
as we have already noticed it is non-zero.
We conclude the bound $h\leq nX^{n+1}$ in this case.
In the latter case, from $s_{h}\vert b_{0}\neq 0$ we 
infer $\vert s_{h}\vert\leq X$, and from \eqref{eq:teile}
we obtain the upper bound $h\leq 2X^{2}$ unless 
$s_{h}a_{1}+b_{0}q=0$. In this case the proof above 
showed $h\vert (s_{h}a_{2}+b_{1}q)$ as well
such that again $h\leq 2X^{2}$ unless $s_{h}a_{2}+b_{1}q=0$. 
We iterate this argument, and by assumption there must 
exist $k\leq n$ with 
$s_{h}a_{k}+b_{k-1}q\neq 0$. Hence we obtain the bound 
$h\leq 2X^{2}\leq nX^{n+1}$ in the second case anyway.
\end{proof}

We recall that by Gelfond's Lemma there exists 
a (small) constant $K=K(n)>0$
such that for $P,Q\in\mathbb{Z}_{\leq n}$ as soon as $H(Q)<KH(P)$
we cannot have that $P$ divides $Q$. This argument was already used
for the proof of~\cite[Theorem~2.3]{bschlei}, which is indeed
very similar to our proof below.

\begin{proof}[Proof of Theorem~\ref{basned}]
Let $\epsilon>0$. As already noticed in the proof 
of Lemma~\ref{mouz} there exist
infinitely many irreducible polynomials 
$P\in \mathbb{Z}_{\leq n}[T]$ with 
$\vert P(\zeta)\vert\leq H(P)^{-w_{n}(\zeta)+\epsilon}$. 
If the degree of $P$ is $n$ the claim follows trivially,
so we may assume it is less. Consider $P$ fixed of large height
and let $X=H(P)\cdot K/2$ with $K=K(n)$ as above. Then 
\begin{equation} \label{eq:track}
H(P)\ll_{n} X, \qquad
\vert P(\zeta)\vert\ll_{n} X^{-w_{n}(\zeta)+\epsilon}.
\end{equation}
Now by definition of $\widehat{w}_{n}(\zeta)$ there exists 
$R\in \mathbb{Z}_{\leq n}[T]$ such that
\begin{equation} \label{eq:trick}
H(R)\leq X, \qquad
\vert R(\zeta)\vert\leq X^{-\widehat{w}_{n}(\zeta)+\epsilon}.
\end{equation}
Obviously $R\neq P$ since $H(P)>X>H(R)$.
By construction in fact
$R$ cannot be a multiple of $P$, and
since $P$ is irreducible $P,R$ have to be coprime.
Again the claim of the theorem follows trivially from \eqref{eq:trick}
if the degree of $R$ is $n$, 
so we may assume $R\in\mathbb{Z}_{\leq n-1}[T]$. 
If $d\leq n-1$ denotes the degree of $R$, then 
let $Q(T)=R(T)T^{n-d}$, which has degree $n$.
Since obviously $P,Q$ are coprime as well and $H(Q)=H(R)\leq X$, 
we can apply the hypothesis to find non-zero
integers $a,b$ of absolute value at most $c(n,\epsilon)X^{\epsilon}$ 
such that $S(T)=aQ(T)+bP(T)$ is irreducible. 
Since $P\in\mathbb{Z}_{\leq n-1}[T]$,
$Q\in\mathbb{Z}_{=n}[T]$ and $a\neq 0$,
we have $S\in\mathbb{Z}_{=n}[T]$.
Clearly $\vert Q(\zeta)\vert\ll_{\zeta,n} \vert Q(\zeta)\vert$.
Thus from \eqref{eq:track} and \eqref{eq:trick} we infer
\[
H(S)\leq \vert a\vert H(Q)+\vert b\vert H(P)\leq \max\{\vert a\vert,
\vert b\vert\} X
\ll_{n,\epsilon} X^{1+\epsilon}
\]
and
\[
\vert S(\zeta)\vert\leq \vert a\vert \cdot\vert Q(\zeta)\vert+
\vert b\vert \cdot\vert P(\zeta)\vert
\leq \max\{\vert a\vert,
\vert b\vert\} X^{-\widehat{w}_{n}(\zeta)+\epsilon}
\ll_{n,\zeta} 
X^{-\widehat{w}_{n}(\zeta)+2\epsilon}.
\]
Hence \eqref{eq:toeroe} follows (conditionally) 
as $\epsilon$ can be chosen arbitrarily small. 

Finally, for \eqref{eq:toeroe1} we readily check that $P,Q$
satisfy the assumptions of Theorem~\ref{irrpol}. 
It yields $a=1$, $b=p$ with 
$\max\{\vert a\vert, \vert b\vert\}=p\leq cX^{\epsilon}$
for which $S=aQ+bP$ has no linear factor. Since $S$ is cubic 
it must in fact be irreducible.
\end{proof}

For the proof of Theorem~\ref{bs}
it is convenient to use the notion of parametric geometry of numbers 
introduced by Schmidt and Summerer~\cite{ss}. We develop the theory only as far as needed
for our concern and slightly modify their notation. We
refer to~\cite{ss} for more details. Keep $\zeta\in\mathbb{R}$ and $n\geq 1$ an integer fixed.
For a parameter $Q>1$ and $1\leq j\leq n+1$, let 
$\psi_{n,j}(Q)$ be the least value of $\eta$ such that
\[
\vert x\vert \leq Q^{1+\eta}, \qquad \vert \zeta^{j}x-y_{j}\vert \leq Q^{-\frac{1}{n}+\eta}
\]
has $j$ linearly independent solutions $(x,y_{1},\ldots,y_{n})\in\mathbb{Z}^{n+1}$. 
Then $-1\leq \psi_{n,j}(Q)\leq 1/n$ for any $Q$
by Minkowski's Theorem. Let
\[
\underline{\psi}_{n,j}= \liminf_{Q\to\infty} \psi_{n,j}(Q), \qquad 
\overline{\psi}_{n,j}= \limsup_{Q\to\infty} \psi_{n,j}(Q).
\]
Similarly denote by $\psi_{n,j}^{\ast}(Q)$ the smallest
number $\eta$ such that
\[
H(P)\leq Q^{\frac{1}{n}+\eta}, \qquad \vert P(\zeta)\vert \leq H(P)^{-1+\eta}
\]
has $j$ linearly independent solutions in $P\in\mathbb{Z}[T]$ of degree at most $n$.
We have $-1/n\leq \psi_{n,j}^{\ast}(Q)\leq 1$ for every $Q>1$. 
Then Mahler's duality, whose special case \eqref{eq:mahler} we mentioned, can be reformulated as
$\vert\psi_{n,j}(Q)+\psi_{n,n+2-j}^{\ast}(Q)\vert \ll 1/\log Q$ 
for $1\leq j\leq n+1$, and hence 
$\underline{\psi}_{n,j}= -\overline{\psi}_{n,n+2-j}^{\ast}$.
It was shown in the remark on page 80 in~\cite{ss} that the identity $\underline{\psi}_{n,1}=-n\overline{\psi}_{n,n+1}$
is equivalent to equality in Khintchine's inequality \eqref{eq:kinschin}, that is
\[
\lambda_{n}(\zeta)=\frac{w_{n}(\zeta)-n+1}{n}.
\]
Recall also the notion of the successive minima exponents $w_{n,j}, \widehat{w}_{n,j}$
defined subsequent to Corollary~\ref{referenzk}.

\begin{proof}[Proof of Theorem~\ref{bs}]      
We will restrict to the case $w<\infty$, the proof of the remaining case $w=\infty$ works very similarly.
By assumption $\zeta\in\mathscr{D}_{n,w}$, for any $\epsilon>0$ 
the estimate
\[
\vert P(\zeta)\vert \leq H(P)^{-w+\epsilon}
\]
has a solution $P(T)=aT+b$ with integers $a,b$ of arbitrarily large height $H(P)=\max\{\vert a\vert, \vert b\vert\}$.
Then the polynomials $P_{0}=P, P_{1}=TP, \ldots, P_{n-1}=T^{n-1}P$ have degree at most $n$, satisfy $H(P_{i})=H(P)$ and 
\[
\vert P_{i}(\zeta)\vert \ll_{n,\zeta} H(P_{i})^{-w+\epsilon}, \qquad 0\leq i\leq n-1. 
\]
Moreover the $P_{i}$ are obviously linearly independent. Thus $w_{n,n}(\zeta)\geq w$, and hence by assumption
$w_{n,1}(\zeta)=w_{n,2}(\zeta)=\cdots=w_{n,n}(\zeta)=w$. This fact can be translated
in the language of the values $\underline{\psi}^{\ast},\overline{\psi}^{\ast}$ defined above as
$-n\underline{\psi}^{\ast}_{n,1}=\overline{\psi}_{n,n+1}^{\ast}$, see the remark in~\cite{ss} and its proof quoted above.
Mahler's duality stated yields the equivalent claim $\underline{\psi}_{n,1}=-n\overline{\psi}_{n,n+1}$. 
Hence there is equality in the right Khintchine inequality \eqref{eq:kinschin}
as carried out above, that is
\[
\lambda_{n}(\zeta)=\frac{w_{n}(\zeta)-n+1}{n}=\frac{w-n+1}{n}.
\]
Thus with Theorem~\ref{korola} we have
\begin{equation} \label{eq:anedobvious}
\widehat{w}_{=n}^{\ast}(\zeta)\geq \frac{1}{\lambda_{n}(\zeta)}=\frac{n}{w-n+1}.
\end{equation}
For the reverse inequality
notice that on the other hand the span of $\{P_{0},\ldots,P_{n-1}\}$
contains only polynomial multiples of $P_{0}$ and thus no irreducible $Q\in \mathbb{Z}_{=n}[T]$ 
(even no irreducible polynomial of degree $2\leq d\leq n$). Thus if
we consider parameters $X$ of the form $X=H(P_{0})$ in \eqref{eq:sukmini2}, we 
conclude that
\begin{equation}  \label{eq:notobvious}
\widehat{w}_{=n}(\zeta)\leq \widehat{w}_{n,n+1}(\zeta). 
\end{equation}
Combination of the left estimate
in the right inequality of \eqref{eq:allesfa}, Mahler's identity \eqref{eq:mahler},
\eqref{eq:anedobvious} and \eqref{eq:notobvious} yields
\[
\frac{n}{w-n+1}= \frac{1}{\lambda_{n}(\zeta)}\leq \widehat{w}_{=n}^{\ast}(\zeta)\leq \widehat{w}_{=n}(\zeta)
\leq \widehat{w}_{n,n+1}(\zeta)=\frac{1}{\lambda_{n}(\zeta)}=\frac{n}{w-n+1}.
\]
Hence \eqref{eq:erstard} follows.
\end{proof}

\begin{remark}
The proof shows that any $\zeta\in\mathscr{D}_{n,w}$
provides equality in the right inequality of \eqref{eq:jodad}.
\end{remark}

For the proof of Theorem~\ref{liouspez} we recall~\cite[Lemma~3.1]{bschlei}, where we drop the originally
involved condition which is easily seen not to be 
required for the conclusion. 

\begin{lemma} \label{bslemma}
Assume $P$ and $Q$ are coprime polynomials of degree $m$ and $n$ respectively,
and $\zeta$ is a real number. Then
at least one of the inequalities
\begin{equation} \label{eq:nedwoa}
\vert P(\zeta)\vert\gg_{m,n,\zeta} H(P)^{-n+1}H(Q)^{-m}, \qquad \vert Q(\zeta)\vert\gg_{m,n,\zeta} H(P)^{-n}H(Q)^{-m+1}
\end{equation}
holds.
\end{lemma}

\begin{proof}[Proof of Theorem~\ref{liouspez}]    
We will again only deal with the case $w<\infty$, the case $w=\infty$ can be treated very similarly using \eqref{eq:unendlich}. 
So let $w\in [2n-1, \infty)$ and $\zeta\in \mathscr{B}_{w}$.
Let us assume $\rho>0$ is fixed and $Q$ is an irreducible polynomial of degree exactly $n$ such that 
\begin{equation} \label{eq:annahme}
\vert Q(\zeta)\vert \leq H(Q)^{-t-\rho}, \qquad t=\frac{nw}{w-n+1}.
\end{equation}
For every convergent $p_{j}/q_{j}$ to $\zeta$
let $P_{j}(T)=q_{j}T-p_{j}$.  
Then, as pointed out in~\cite{bug}, we have $\vert P_{j}(\zeta)\vert \asymp_{n,\zeta} H(P_{j})^{-w}$ and
$H(P_{j+1})\asymp_{n,\zeta} H(P_{j})^{w}$ for $j\geq 1$.
Let $i$ be the index for which $H(P_{i})=q_{i}\leq H(Q)<q_{i+1}=H(P_{i+1})$, where we used $\zeta\in(0,1)$.
Clearly $P_{j}$ is coprime to $Q$ for all $j\geq 1$, since $P_{j}$ have degree one and $Q$ is irreducible of degree $n\geq 2$.
Thus we can apply Lemma~\ref{bslemma} with $m=1, n$ and the pair of polynomials $P_{j},Q$. 
Let $\delta>0$. In case of $H(Q)\leq H(P_{i})^{w-n+1-\delta}$,
for $j=i$ the left inequality of \eqref{eq:nedwoa} is violated as it would lead to
\begin{align*}
H(P_{i})^{-w+\delta}&= H(P_{i})^{-n+1}H(P_{i})^{-(w-n+1-\delta)}\\
&\leq H(P_{i})^{-n+1}H(Q)^{-1}\ll_{n,\zeta} 
\vert P_{i}(\zeta)\vert \ll_{n,\zeta} H(P_{i})^{-w},
\end{align*}
contradiction for large $i$. Thus we must have
$\vert Q(\zeta)\vert\gg_{n,\zeta} H(P_{i})^{-n}\geq H(Q)^{-n}$, contradicting the assumption \eqref{eq:annahme} 
for large $i$ since $t\geq n$. 
If otherwise $H(Q)\geq H(P_{i})^{w-n+1-\delta}$, then we apply Lemma~\ref{bslemma} for the polynomials $P_{i+1}$ and $Q$.
The left inequality in \eqref{eq:nedwoa} leads to
\[
H(P_{i+1})^{-w}\gg_{n,\zeta} \vert P_{i+1}(\zeta)\vert\gg_{n,\zeta} H(P_{i+1})^{-n+1}H(Q)^{-1}\geq  H(P_{i+1})^{-n}
\]
contradiction to $w>n$ for large $i$. Similarly the right inequality in \eqref{eq:nedwoa} leads to
\[
\vert Q(\zeta)\vert\gg_{n,\zeta} H(P_{i+1})^{-n}\gg_{n,\zeta} H(P_{i})^{-nw}\gg_{n,\zeta} H(Q)^{-nw/(w-n+1-\delta)}
\]
again a contradiction to \eqref{eq:annahme} for large $i$ if $\delta$ was chosen small enough that we still
have $t+\rho>nw/(w-n+1-\delta)$. Hence there can only be finitely many solutions to \eqref{eq:annahme} for any $\rho>0$
and irreducible $Q\in\mathbb{Z}_{=n}[T]$.
The claim \eqref{eq:holteufel} follows.
\end{proof}

For the deduction of Theorem~\ref{ce}, 
we apply the identities \eqref{eq:kannt} for $\zeta\in\mathscr{B}_{w}$. 
In fact the lower bound
\begin{equation} \label{eq:donduck}
\widehat{w}_{n}^{\ast}(\zeta)\geq \frac{w_{n}(\zeta)}{w_{n}(\zeta)-n+1}
= \frac{w}{w-n+1}
\end{equation}
established
by Bugeaud and Laurent~\cite[Theorem~2.1]{buglaur},
would suffice.

\begin{proof}[Proof of Theorem~\ref{ce}]
Combination of \eqref{eq:erstard} with \eqref{eq:kannt} yields 
$\widehat{w}_{=n}(\zeta)= n/(w-n+1)<n=\widehat{w}_{n}(\zeta)$, as soon as $w>n$. 
Similarly, from \eqref{eq:kannt} we infer $\widehat{w}_{=n}^{\ast}(\zeta)=n/(w-n+1)<w/(w-n+1)=\widehat{w}_{n}^{\ast}(\zeta)$.
Thus we have shown \eqref{eq:lucia}.
For \eqref{eq:bleiben}, again \eqref{eq:kannt}
implies $\widehat{w}_{n}^{\ast}(\zeta)=w/(w-n+1)>1$ strictly, 
as soon as $w<\infty$. On the other hand, when $w\geq 2n-1$, we readily check $\widehat{w}_{=n}(\zeta)= n/(w-n+1)\leq 1$,
and similarly $\widehat{w}_{=m}(\zeta)\leq 1$ for $1\leq m\leq n$.
The remaining estimates of \eqref{eq:bleiben} are obvious consequences of \eqref{eq:wepper}, \eqref{eq:bbb} and \eqref{eq:allesfa}
and the previous observation.

When $w>2n-1$, from \eqref{eq:holteufel} we infer $w_{=n}^{\ast}(\zeta)\leq w_{=n}(\zeta)\leq nw/(w-n+1)< 2n-1<w
= w_{n}^{\ast}(\zeta)=w_{n}(\zeta)$, which shows the two most left inequalities of \eqref{eq:allestickln}.
The uniform inequalities in \eqref{eq:allestickln} were already established in \eqref{eq:lucia} under the weaker condition $w>n$.
\end{proof}

\begin{proof} [Proof of Theorem~\ref{prof}]
We have $\widehat{w}_{n,n}(\zeta)\geq 1$ for any real 
$\zeta$ not algebraic of degree at most $n$. More
generally the analogous exponent assigned to any
$\underline{\zeta}\in\mathbb{R}^{n}$ that is $\mathbb{Q}$-linearly
independent together with $\{1\}$ is bounded below by $1$.
This follows from the results in~\cite{ss}. On the other hand, 
$\widehat{w}_{=n}^{\ast}(\zeta)=\widehat{w}_{=n}(\zeta)=n/(w-n+1)<1$ 
for $\zeta\in\mathscr{D}_{n,w}$ when $w>n$, by Theorem~\ref{bs}. 
Combination shows the first claim.
For the second assertion,
we first claim $w_{n,n}(\zeta)=w_{n}(\zeta)=w$. 
For any convergent $p/q$ to $\zeta\in\mathscr{D}_{n,w}$
consider $P(T)=qT-p$ and the derived linearly independent polynomials 
$\{D_{1},D_{2},\ldots,D_{n}\}=\{ P,TP,\ldots,T^{n-1}P\}$. 
We clearly have $H(D_{j})=H(P)$
and $\vert D_{j}(\zeta)\vert \asymp_{n,\zeta} \vert P(\zeta)\vert$. This shows $w_{n,n}(\zeta)\geq w_{1}(\zeta)=w$,
the reverse inequality $w=w_{n}(\zeta)\geq w_{n,n}(\zeta)$ is obvious.
On the other hand, from Theorem~\ref{liouspez} we obtain
$w_{=n}^{\ast}(\zeta)\leq w_{=n}(\zeta)=nw/(w-n+1)<w$ for
$\zeta\in\mathscr{D}_{n,w}$ with $w>2n-1$. This concludes the proof
of the second claim.   
\end{proof}

We prove Theorem~\ref{algint}, 
in a similar way as Theorem~\ref{bs}.

\begin{proof}[Proof of Theorem~\ref{algint}]
Notice that the assumptions are precisely as in Theorem~\ref{bs}. For
the left inequality, as in the proof of Theorem~\ref{bs} we obtain
$\lambda_{n-1}(\zeta)=(w-(n-1)+1)/(n-1)$ for 
$\zeta\in\mathscr{D}_{n,w}$ 
since $w\geq n \geq n-1$ (index shift $n$ to $n-1$ compared 
to Theorem~\ref{bs}). 
Thus \eqref{eq:dregsrr} indeed yields
\[
\widehat{w}_{n}^{\ast int}(\zeta) \geq \frac{1}{\lambda_{n-1}(\zeta)}= \frac{n-1}{w-n+2}.
\]

The most right inequality of \eqref{eq:lebal} remains to be proved.
For simplicity put $v=n/(w-n+1)$.
In the proof of Theorem~\ref{bs} we noticed that 
$\zeta\in\mathscr{D}_{n,w}$ satisfies $\widehat{w}_{n,n+1}(\zeta)=v$.
More precisely, the proof showed that for any $\epsilon>0$ there are arbitrarily large parameters
$X$ such that every solution $P\in \mathbb{Z}_{\leq n}[T]$ of 
\begin{equation} \label{eq:phelps}
H(P)\leq X, \qquad \vert P(\zeta)\vert \leq X^{-v-\epsilon}
\end{equation}
is a polynomial multiple of a linear polynomial $Q(T)=aT+b$.
Here $-b/a$ is a very good rational approximation (in particular a convergent) to $\zeta$. 
By elementary facts on continued fractions we clearly have $a>1$ 
and $(a,b)=1$. 
It follows from the Lemma of Gau\ss\
that every polynomial multiple $U(T)=R(T)Q(T)$ of $Q$ with arbitrary 
$R\in\mathbb{Q}[T]$ which has integral coefficients
$U\in\mathbb{Z}[T]$, must actually arise from $R\in\mathbb{Z}[T]$. Thus $U(T)$ has 
leading coefficient divisible by $a$ and hence is not monic. 
In other words, for parameters $X$ as above every monic polynomial 
$P\in \mathbb{Z}_{\leq n}[T]$ with $H(P)\leq X$ must satisfy
\[
\vert P(\zeta)\vert \geq X^{-v-\epsilon}.
\]
The right inequality in \eqref{eq:lebal} follows as we may let $\epsilon$ tend to $0$.
The equivalence claim \eqref{eq:propella}
for Liouville numbers follows immediately 
from the upper and lower bound in \eqref{eq:lebal}.
\end{proof}

The proof more precisely shows the finiteness of solutions 
$P\in\mathbb{Z}_{\leq n}[T]$ to \eqref{eq:phelps} with bounded leading coefficient when $\zeta\in\mathscr{D}_{n,w}$.

\section{Some open problems}  \label{openp}

In this section we formulate selected open problems, mainly concerning our new exponents 
for approximation of exact degree.
Some of them have already been addressed, explicitly or implicitly, in the course of the paper.
First we discuss several variants of Wirsing's problem 
which we introduced right at the beginning in Section~\ref{out}.

\begin{problem} \label{pr1}
Is it true that for any transcendental real $\zeta$ and every $n\geq 3$ we have $w_{=n}^{\ast}(\zeta)\geq n$?
Does even the estimation
\begin{equation} \label{eq:spirit}
\vert \zeta-\alpha\vert \ll_{n,\zeta} H(\alpha)^{-n-1}
\end{equation}
have infinitely many solutions $\alpha\in \mathbb{A}_{=n}$?
Similarly, is it true that for every $n\geq 3$ we have $w_{n+1}^{\ast int}(\zeta)\geq n$, or more generally
 $w_{=m}^{\ast int}(\zeta)\geq n$ for every $m\geq n+1$?
What about refinements in the spirit of \eqref{eq:spirit}?
\end{problem}  

Recall we have shown \eqref{eq:spirit} for $n=2$ in Theorem~\ref{bugpro}, whereas
$w_{3}^{\ast int}(\zeta)<2$ for certain extremal numbers was pointed out in \eqref{eq:damienr}. Next we discuss variants of the 
related natural question discussed in Section~\ref{natu}.

\begin{problem} \label{pro2}
Do we have $w_{=n}(\zeta)\geq n$ for all $n\geq 4$ and any transcendental real number $\zeta$?
Is it even true that the inequality
\begin{equation} \label{eq:bahamas}
\vert P(\zeta)\vert \ll_{n,\zeta} H(P)^{-n}
\end{equation}
has infinitely many solutions $P\in \mathbb{Z}_{=n}[T]$?
What about $w_{=n+1}^{int}(\zeta)\geq n$ for $n\geq 3$?
\end{problem}

As pointed out we strongly believe the answer to be positive
at least for $w_{=n}(\zeta)$. We cannot prove the stronger
condition \eqref{eq:bahamas} even for $n=3$, for
if $\widehat{w}_{3}(\zeta)=3$ the method of Theorem~\ref{basned} 
would require a bound of order $O(1)$ for the smallest suitable 
prime $p$ in the auxiliary Theorem~\ref{irrpol}.
On the other hand, observe that 
$\widehat{w}_{=n}(\zeta)<n$ holds for certain $\zeta$, 
as follows from Theorem~\ref{bs}.
In Theorem~\ref{ce} we saw that the exponents of bounded degree
can differ vastly from the exponents of exact degree.
However, we may ask for a generalization of Theorem~\ref{grad2}.

\begin{problem} \label{p3e}
Assume $n\geq 3$ is an integer and $\zeta$ is a transcendental real number with $\widehat{w}_{n}(\zeta)>n$. 
Is it true that
\[
w_{=n}(\zeta)= w_{n}(\zeta), \qquad \widehat{w}_{=n}(\zeta)= \widehat{w}_{n}(\zeta), \qquad w_{=n}^{\ast}(\zeta)=w_{n}^{\ast}(\zeta),
\qquad \widehat{w}_{=n}^{\ast}(\zeta)= \widehat{w}_{n}^{\ast}(\zeta)
\]
necessarily holds? Further, is it true that $\zeta$ cannot
be a $U$-number?
\end{problem}

The claim could potentially be true in a trivial sense in case no 
number satisfies the condition.
For $n\geq 3$ we cannot rule out that $\zeta$ is a $U_{m}$-number of index $2\leq m\leq n-1$, which is an empty 
range for $n=2$ as in Theorem~\ref{grad2}.
The next question concerns the relation between approximation by algebraic numbers
versus algebraic integers.

\begin{problem} \label{problem4}
Let $n\geq 1$ be an integer.
Does there exist transcendental real $\zeta$ such 
that $w_{=n}(\zeta)<w_{=n+1}^{int}(\zeta)$ or $w_{=n}^{\ast}(\zeta)<w_{=n+1}^{\ast int}(\zeta)$?
Similarly for $\widehat{w}_{=n}(\zeta)<\widehat{w}_{=n+1}^{int}(\zeta)$ or $\widehat{w}_{=n}^{\ast}(\zeta)<\widehat{w}_{=n+1}^{\ast int}(\zeta)$.
More general, determine the spectra of $w_{=n}(\zeta)-w_{=n+1}^{int}(\zeta), w_{=n}^{\ast}(\zeta)-w_{=n+1}^{\ast int}(\zeta),
\widehat{w}_{=n}(\zeta)-\widehat{w}_{=n+1}^{int}(\zeta)$ and $\widehat{w}_{=n}^{\ast}(\zeta)-\widehat{w}_{=n+1}^{\ast int}(\zeta)$.
\end{problem}

The estimate \eqref{eq:damienr} for some extremal numbers showed that $w_{=n+1}^{int}(\zeta)<n$ is possible, at least for $n=2$.
It seems that conversely numbers which are very well approximable by algebraic integers have not been constructed yet for any degree.

\begin{problem}
For $n\geq 1$, construct real transcendental $\zeta$ for which $w_{=n+1}^{int}(\zeta)>n$, or even $w_{=n+1}^{\ast int}(\zeta)>n$.
\end{problem}

The next problem is much more general and a complete answer seems out of reach.

\begin{problem}
Determine the spectra of the new exponents $w_{=n}, w_{=n}^{\ast}, \widehat{w}_{=n},\ldots$.
\end{problem}

We have noticed that combination of Corollary~\ref{sb} and Corollary~\ref{najor} yields
that a transcendental real number is a Liouville number if and only if
$\widehat{w}_{=n}(\zeta)=0$, or equivalently $\widehat{w}_{=n}^{\ast}(\zeta)=0$, for all $n\geq 1$.
Recall also the characterization \eqref{eq:propella} for Liouville numbers.
A natural related question for the exponents $\widehat{w}_{n}^{\ast}(\zeta)$ remains partly open.

\begin{problem} \label{prob}
Is a transcendental real number $\zeta$ a Liouville number if and only if
\begin{equation} \label{eq:full}
\widehat{w}_{n}^{\ast}(\zeta)=1, \qquad  n\geq 1,
\end{equation}
holds?
\end{problem}

As stated in Section~\ref{bounded}, we have $\widehat{w}_{n}^{\ast}(\zeta)\geq 1$ and any Liouville number
has the property \eqref{eq:full}.
On the other hand, the estimate \eqref{eq:donduck} implies
that $w_{2}(\zeta)=\infty$ is necessary for \eqref{eq:full}.
We conclude that $\zeta$ must be either a Liouville number or a $U_{2}$-number.
Our next problem is motivated by Theorem~\ref{algint}.

\begin{problem}
Determine $\widehat{w}_{n}^{int}(\zeta)$ and $\widehat{w}_{n}^{\ast int}(\zeta)$
for $\zeta\in\mathscr{D}_{n,w}$, 
or at least for the special examples $\zeta\in\mathscr{B}_{w}$.
\end{problem}

We have noticed below Theorem~\ref{algint} that it is plausible to believe in equality with the left bound in \eqref{eq:lebal}.
We conclude with another problem which stems from Theorem~\ref{algint}, concerning inhomogeneous approximation.
We only state the case $n=1$ explicitly.

\begin{problem} \label{probo}
Let $\zeta$ be a Liouville number. For $\alpha\in\mathbb{R}$ 
denote by $\widehat{w}_{1}(\zeta,\alpha)$ the supremum of exponents $w$ for which
\[
1\leq x_{1} \leq X, \qquad  \vert \alpha+x_{0}+\zeta x_{1}\vert \leq X^{-w}
\]
has a solution in integers $x_{0},x_{1}$ for all large $X$. Does 
the spectrum of $\widehat{w}_{1}(\zeta,\alpha)$ contain (or even equal) the interval $[0,1]$?

\end{problem}

As remarked above it follows from~\cite{bula} that $\widehat{w}_{1}(\zeta,\alpha)=0$ for almost all $\alpha$, and by our results,
including any $\alpha$ of the form $Q(\zeta)$ with $Q\in\mathbb{Q}_{\geq 2}[T]$. Moreover $\{1\}$ is contained in 
spectrum, and we may take $\alpha=Q(\zeta)$ with any $Q\in\mathbb{Q}_{\leq 1}[T]$. In contrast to
$\widehat{w}_{1}(\zeta)=\widehat{w}_{1}(\zeta,0)=1$ 
for all $\zeta$, it seems that $\widehat{w}_{1}(\zeta,\alpha)>1$
cannot be excluded for arbitrary $\alpha$ with the current knowledge.

See also~\cite[Section~10.2]{bugbuch} and~\cite{bdraft} for
several problems concerning the classic exponents $w_{n}, \widehat{w}_{n}, w_{n}^{\ast}, \widehat{w}_{n}^{\ast}, \lambda_{n}, \widehat{\lambda}_{n}$
(some questions of the first reference have already been solved).

\vspace{0.5cm}

Many 
thanks to Yann Bugeaud for providing references concerning Theorem~\ref{korola} and Theorem~\ref{ppp},
and to Damien Roy for help with the presentation of the results!


\begin{thebibliography}{99}

\bibitem{adambug} B. Adamczewski and Y. Bugeaud,  Mesures de transcendance et aspects quantitatifs de la m\'{e}thode de Thue-Siegel-Roth-Schmidt.
{\em Proc. London Math. Soc.} 101 (2010), 1--31.

\bibitem{alni} K. Alnia\c{c}ik, On Mahler's $U$-numbers. {\em Amer. J. Math.} 105 (1983), no. 6, 1347--1356. 

\bibitem{apostol} T. M. Apostol, Introduction to analytic number theory, {\em Undergraduate Texts in Mathematics}, 
New York-Heidelberg, Springer press (1976).

\bibitem{baks} A. Baker and W.M. Schmidt, Diophantine approximation and Hausdorff dimension,
{\em Proc. London Math. Soc.} 21 (1970), 1--11.





\bibitem{bernik} V. I. Bernik, Application of the Hausdorff dimension in the theory of Diophantine approximations,
{\em Acta Arith.} 42 (1983), 219--253 (in Russian). English translation in Amer. Math. Soc. Transl. 140 (1988), 15--44.

\bibitem{bonciocat} N.C. Bonciocat,
Upper bounds for the number of factors for a class of polynomials with rational coefficients, {\em Acta Arith.} 2004, Vol.113(2), 175--187.

\bibitem{bugagain} N.C. Bonciocat, Y. Bugeaud, M. Cipu and M. Mignotte,
Irreduciblity criteria for sums of two relatively prime polynomials,
{\em Publ. Math. Debrecen} 87 (2015), no. 3--4, 255--267. 

\bibitem{bugbuch} Y. Bugeaud, Approximation by Algebraic Numbers,
{\em Cambridge Tracts in Mathematics} 160 (2004), Cambridge University Press.

\bibitem{bug} Y. Bugeaud, On simultaneous rational approximation to a real numbers and its integral powers, 
{\em Ann. Inst. Fourier (Grenoble)} 60 (2010), 2165--2182.   

\bibitem{bdraft} Y. Bugeaud, Exponents of Diophantine approximation.  
{\em Dynamics and Analytic Number Theory} 96--135, 
London Math. Soc. Lecture Note Ser., 437, Cambridge Univ. Press, Cambridge, 2016. 



\bibitem{buglaur} Y. Bugeaud and M. Laurent, Exponents of Diophantine approximation
and Sturmian continued fractions, {\em Ann. Inst. Fourier (Grenoble)} 55 (2005), no. 3, 773--804.

\bibitem{bula} Y. Bugeaud and M. Laurent, Exponents of homogeneous and inhomogeneous Diophantine approximation,
{\em Moscow Math. J.} 5 (2005), 747--766.

\bibitem{buglau} Y. Bugeaud and M. Laurent, Exponents in Diophantine approximation, 
{\em Diophantine geometry}, 101--121, CRM Series, 4, Ed. Norm., Pisa, 2007.

\bibitem{bschlei} Y. Bugeaud and J. Schleischitz, On uniform approximation to real numbers,
{\em Acta Arith.} 175 (2016), 255--268.

\bibitem{teu} Y. Bugeaud and O. Teuli\'{e}, Approximation d'un nombre r\'{e}el par des nombres algebriques
de degr\'{e} donn\'{e}, {\em Acta Arith.} 93 (2000), no. 1, 77--86.

\bibitem{cassels} J. W. S. Cassels, An introduction to Diophantine approximation, {\em Cambridge Tracts in Math. and Math. Phys.},
vol. 99, Cambridge University Press, 1957.

\bibitem{cav}  M. Cavachi,
On a special case of Hilbert’s irreducibility theorem,
{\em J. Number Theory 82} (2000), 96--99.

\bibitem{cav2}   M. Cavachi, M. V\^{a}j\^{a}itu and A. Zaharescu,
A class of irreducible polynomials,
{\em J. Ramanujan Math. Soc.} 17
(2002), 161--172.

\bibitem{davs} H. Davenport and W. M. Schmidt, 
Approximation to real numbers by quadratic irrationals,
{\em Acta Arith.} 13 (1967), 169--176.


\bibitem{davsh} H. Davenport and W. M. Schmidt, 
Approximation to real numbers by algebraic integers, 
{\em Acta Arith.} 15 (1969), 393--416.


\bibitem{ogerman} O. German, On Diophantine exponents and Khintchine's transference principle, 
{\em Mosc. J. Comb. Number Theory} 2 (2012), 22--51.


\bibitem{hw} G. H. Hardy and E. M. Wright, An Introduction to the Theory of Numbers, 5th ed. Oxford, England: Clarendon Press, 354--358 (1979).

%\bibitem{ja1} V. Jarn\'ik, Une remarque sur les approximations diophantiennes lin\'{e}aires, {\em Acta Sci.
%Math. Szeged} 12 (1950), pars B, 82--86.

%\bibitem{ja2} V. Jarn\'ik, Contribution to the theory of homogeneous linear Diophantine approximations
%, {\em Czechoslovak Math. J.} 4 (79) (1954), 330--353 (in Russian).


\bibitem{khint} A. Y. Khintchine, Zur metrischen Theorie der diophantischen Approximationen,
{\em Math. Z.} 24 (1926), 706--714.

\bibitem{kinn} A. Y. Khintchine, \"Uber eine Klasse linearer diophantischer Approximationen,
{\em Rendiconti Circ. math. Palermo} 50 (1926), 170--195.



\bibitem{lau} M. Laurent, Simultaneous rational approximation to successive powers
of a real number, {\em Indag. Math.} 11 (2003), 45--53. 

\bibitem{leveque} W. J. LeVeque, On Mahler's U-numbers, {\em J. London Math. Soc.} 28 (1953), 220--229.

\bibitem{minkow} H. Minkowski, Geometrie der Zahlen, {\em Teubner, Leipzig,} 1910.

\bibitem{nimo} N. Moshchevitin, A note on two linear forms, {\em Acta Arith.} 162 (2014), 43--50.

\bibitem{mosh} N. Moshchevitin, Positive integers: counterexamples to W.M. Schmidt's conjecture, 
{\em Mosc. J. Comb. Number Theory} 2 (2012), no. 2, 63--84.

\bibitem{roy2} D. Roy, Approximation by algebraic integers II, {\em Ann. Math.} 158 (2003), 1081--1087.

\bibitem{roy} D. Roy, Approximation by algebraic integers I, {\em Proc. London Math. Soc.} 88 (2004), 42--62.



\bibitem{royexp} D. Roy, On two exponents of approximation related to a real number and its square,
{\em Canad. J. Math.} 59 (2007), 211--224. 

\bibitem{damiroy} D. Roy,  Simultaneous approximation to a real number, its square and its cube, {\em Acta Arith.} 133 (2008), 185--197.

\bibitem{royneu} D. Roy, Diophantine approximation with sign constraints, {\em Monatsh. Math.} 173 (2014), no. 3, 417--432.


\bibitem{j2} J. Schleischitz, Two estimates concerning classical Diophantine approximation constants,
{\em Publ. Math. Debrecen} 84/3-4 (2014), 415--437. 

\bibitem{j3} J. Schleischitz, On the spectrum of Diophantine approximation constants,
{\em Mathematika} 62 (2016), 79--100.



\bibitem{ichrelation} J. Schleischitz, On simultaneous approximation to successive powers of a real number, {\em Indag. Math.} 28 (2017), no. 2, 406--423.

\bibitem{ichlondon} J. Schleischitz, Approximation to an extremal number, its square and its cube,
{\em Pacific J. Math.} 287 (2017), no. 2, 485--510.



\bibitem{ichstw} J. Schleischitz, Cubic approximation to Sturmian continued fractions, {\em J. Number Theory} 184 (2018), 270--299.

\bibitem{ichspec} J. Schleischitz, An equivalence principle
between polynomial and simultaneous Diophantine
approximation, {\em arXiv: 1704.00055}.

%\bibitem{ichcomments} J. Schleischitz, On the discrepancy between best and uniform approximation, % fehlt!!!!!!!!!!!!!!!!!!!!!!!!!!!!!!!!!!!!!!!!!!!!!!!!

\bibitem{wschmidt} W.M. Schmidt, Two questions in Diophantine approximation,
{\em Monatsh. Math.} 82 (1976), no. 3, 237--245.

\bibitem{512} W.M. Schmidt, Diophantine approximation, {\em Lecture Notes in Math.}, 785, Springer, Berlin (1980).



\bibitem{ss} W.M. Schmidt and L. Summerer, Parametric geometry of numbers and applications, 
{\em Acta Arith.} 140 (2009), no. 1, 67--91.

\bibitem{teuallein} O. Teulie, Approximation d'un nombre r\'{e}el par des unit\'{e}s alg\'{e}briques (French), 
{\em Monatsh. Math.} 132 (2001), no. 2, 169--176. 

\bibitem{Tsi07} K.I. Tsishchanka,
On approximation of real numbers by algebraic numbers of bounded degree,
{\em J. Number Theory} 123 (2007), 290--314.


\bibitem{wirsing} E. Wirsing, Approximation mit algebraischen Zahlen beschr\"ankten Grades,
{\em J. Reine Angew. Math.} 206 (1961), 67--77.

\end{thebibliography}
\end{document}